\def\NN{}
\documentclass[10pt]{amsart}
\usepackage[top=1 in,bottom=1 in, left=1 in, right=1 in, includehead,includefoot]{geometry}

\numberwithin{equation}{section}
\usepackage{amsmath,amssymb,amsthm, dsfont}
\usepackage{enumerate}
\usepackage[numbers,sort&compress]{natbib}
\usepackage{graphicx}
\usepackage{color}
\usepackage{url}

\newcommand{\cM}{{\mathcal M}}

\newcommand{\cY}{{\mathcal Y}}

\newcommand{\E}{{\mathbb E}}

\newcommand{\N}{{\mathbb N}}

\renewcommand{\P}{{\mathbb P}}

\newcommand{\R}{{\mathbb R}}

\newcommand{\Z}{{\mathbb Z}}

\newcommand{\Var}{{\rm Var}}

\newcommand{\ind}{\mathds{1}}

\newtheorem{theorem}{Theorem}[section]
\newtheorem{proposition}{Proposition}[section]
\newtheorem{lemma}{Lemma}[section]

\theoremstyle{definition}

\newtheorem{remark}{Remark}[section]
\newtheorem{assumption}{Assumption}[section]

\newcommand{\comment}[1]{}
\def\limn{\lim_{n\to\infty}}
\def\limsupn{\limsup_{n\to\infty}}
\def\mfa{\text{ for all }}
\def\vv#1{{\boldsymbol #1}}

\def\indd#1{{\ind}_{\{#1\}}}
\def\inddd#1{{\ind}_{\left\{#1\right\}}}
\def\nn#1{\left\|#1\right\|}
\def\snn#1{\|#1\|}
\def\inv{^{-1}}

\def\proba{\mathbb P}
\def\esp{\mathbb E}
\def\pp#1{\left(#1\right)}
\def\spp#1{(#1)}
\def\bb#1{\left[#1\right]}
\def\ccbb#1{\left\{#1\right\}}
\def\abs#1{\left|#1\right|}
\def\Rd{{\mathbb R^d}}
\def\Nd{{\mathbb N^d}}
\def\summ#1#2#3{\sum_{#1=#2}^{#3}}
\def\sif#1#2{\sum_{#1=#2}^\infty}
\def\vif#1#2{\bigvee_{#1=#2}^\infty}
\def\prodd#1#2#3{\prod_{#1=#2}^{#3}}
\def\mmid{\;\middle\vert\;}
\def\wt#1{\widetilde{#1}}
\def\what#1{\widehat{#1}}
\def\wb#1{\overline{#1}}
\def\topp#1{^{(#1)}}
\def\qmand{\quad\mbox{ and }\quad}

\def\qmwith{\quad\mbox{ with }\quad}
\def\mwith{\mbox{ with }}
\def\calA{{\mathcal A}}
\def\calB{{\mathcal B}}
\def\calE{{\mathcal E}}
\def\calF{{\mathcal F}}
\def\calG{{\mathcal G}}
\def\calJ{{\mathcal J}}
\def\calK{{\mathcal K}}
\def\calL{{\mathcal L}}

\def\calM{{\mathcal M}}
\def\calN{{\mathcal N}}
\def\calR{{\mathcal R}}
\def\calY{{\mathcal Y}}

\def\eqd{\stackrel{d}=}
\def\weakto{\Rightarrow}

\newcommand{\eqnh}{\begin{eqnarray*}}
\newcommand{\eqne}{\end{eqnarray*}}
\newcommand{\eqnhn}{\begin{eqnarray}}
\newcommand{\eqnen}{\end{eqnarray}}
\newcommand{\equh}{\begin{equation}}
\newcommand{\eque}{\end{equation}}

\def\Cabf{\mathfrak C_{\alpha,\beta}(f)}
\def\Cabfs{\mathfrak C_{\alpha,\beta,*}(f)}
\def\mab{\calM_{\alpha,\beta}}

\def\mabZ{\calM_{\alpha,\beta,Z}}
\def\mabl{\calM_{\alpha,\beta}^{\rm lo}}
\def\ddelta#1{\delta_{\pp{#1}}}

\def\SM{{\rm SM}}
\def\PPP{{\rm PPP}}
\def\supp{{\rm supp}}

\author{Olivier Durieu}
\address{
Olivier Durieu\\
Institut Denis Poisson, UMR-CNRS 7013\\
Universit\'e de Tours, Parc de Grandmont, 37200 Tours, France.
}
\email{olivier.durieu@univ-tours.fr}

\author{Yizao Wang}
\address
{
Yizao Wang\\
Department of Mathematical Sciences\\
University of Cincinnati\\
2815 Commons Way\\
Cincinnati, OH, 45221-0025, USA.
}
\email{yizao.wang@uc.edu}


\title[Phase transition for extremes]{Phase transition for extremes of a stochastic model with long-range dependence and multiplicative noise}

\begin{document}\sloppy
\begin{abstract}
We consider a stochastic process with long-range dependence perturbed by multiplicative noise. The marginal distributions of both the original process and the noise have  regularly-varying tails, with tail indices $\alpha,\alpha'>0$, respectively. The original process is taken as the regularly-varying Karlin model, a recently investigated model that has long-range dependence characterized by a memory parameter $\beta\in(0,1)$. We establish limit theorems for the extremes of the model, and reveal a phase transition. In terms of the limit there are three different regimes: signal-dominance regime $\alpha<\alpha'\beta$, noise-dominance regime $\alpha>\alpha'\beta$, and critical regime $\alpha = \alpha'\beta$. As for the proof, we actually establish the same phase-transition phenomena for the so-called Poisson--Karlin model with multiplicative noise defined on generic metric spaces, and apply a Poissonization method to establish the limit theorems for the one-dimensional case as a consequence.
\end{abstract}
\maketitle
\NN

\section{Introduction}
\subsection{Background}
The motivating example of this paper concerns the following model of stationary stochastic processes
\equh\label{eq:1}
X_i = \sigma_iZ_i, 
\quad i\in\N :=\{1,2,\dots\},
\eque
where $\{\sigma_i\}_{i\in\N}$ is a stationary sequence of random variables and $\{Z_i\}_{i\in\N}$ are i.i.d.~copies of certain random variable $Z$, the two sequences being independent. 
We are mostly interested in characterizing scaling limits for the model 
\eqref{eq:1}
 from a theoretical point of view. It turned out that, despite its simple structure, it may exhibit an intriguing phase transition in terms of the limit of extremes.  
 We refer to the process $\{\sigma_i\}_{i\in\N}$ as the {\em signal process}, and $\{Z_i\}_{i\in\N}$ as the {\em multiplicative noise}  (or the {\em noise process}).

For \eqref{eq:1}, 
we investigate the case that the perturbed process $X = \{X_i\}_{i\in\N}$ has regularly-varying tails with index $\gamma>0$ (i.e.~$\wb F_{X_1}(x):=\proba(X_1>x)\in RV_{-\gamma}$). 
Our first assumption is that both $\sigma_1$ and $Z_1$ have regularly-varying tails, with index $\alpha>0$, $\alpha'>0$, respectively. Indeed, if we consider a single random variable $X_1 = \sigma_1Z_1$, it is a well-known result due to \citet{breiman65some} (see \citep{jessen06regularly} for more references) that then $X_1$ has a regularly-varying tail with the dominant index of the two ($\gamma = \min\{\alpha,\alpha'\}$). (Strictly speaking, when say the tail of $\sigma_1$ dominates, we do not need to assume $Z_1$ to have a regularly-varying tail for $X_1$, but simply that $\esp Z_1^{\alpha+\epsilon}<\infty$. For the sake of simplicity we restrict our discussions here to the regularly-varying tails, while our main results later are proved under more general assumptions.)

Second, at the process level, we are interested in the case that the signal process $\{\sigma_i\}_{i\in\N}$ is with long-range dependence, and we take the recently introduced heavy-tailed  (regularly varying) {\em Karlin model} \citep{durieu20infinite,durieu18family} 
for $\sigma$. 
 Heuristically, there is a memory parameter $\beta\in(0,1)$ in the models of our interest, and by long-range dependence we mean that 
 the scaling limit for extremes of $\sigma$ is of abnormal order, depending on $\beta$, compared to a sequence of i.i.d.~random variables of the same marginal. Moreover, certain {\em long-range clustering} of extremes appear in the limit \citep{samorodnitsky16stochastic}.
This is in stark contrast to most time-series models investigated so far in the literature, which exhibit possibly 
{\em local clustering} of extremes 
(a.k.a.~extremal clustering in the literature)
in the limit.
Local clustering is a feature of microscopic behaviors: it is usually quantified by the {\em extremal index} taking values from $(0,1]$, interpreted as the reciprocal of the expected size of the extreme cluster (index equal to one meaning no clustering), and more precisely described by limit theorems for multivariate regular variations and tail processes. Extremes with local clustering have been extensively investigated in extreme-value theory (e.g.~\citep{aldous89probability,leadbetter83extremes} on early developments and~\citep{basrak18invariance,dombry18tail,janssen19spectral,basrak09regularly} for recent advances).  Long-range clustering, on the other hand, is a feature of macroscopic behaviors 
 (or, it could be thought of local clustering with unbounded size),
and very few examples have been worked out (e.g.~\citep{durieu18family,lacaux16time,samorodnitsky19extremal}).

Now we describe the heavy-tailed (power-law) {\em Karlin model} \citep{karlin67central,durieu20infinite} 
for the signal process $\sigma$. The fact that this model exhibits long-range clustering of extremes was demonstrated in \citep{durieu18family}. 
Consider a probability measure on $\N$ with mass function $\{p_\ell\}_{\ell\in\N}$, and for the sake of simplicity assume $p_\ell\sim C\ell^{-1/\beta}$ where $\beta\in(0,1)$ is the memory parameter (see \eqref{eq:nu} below for the exact assumption). Let $\{Y_i\}_{i\in\N}$ be i.i.d.~sampling from $\N$ according to $\proba(Y_i = \ell) = p_\ell$. Let $\{\varepsilon_\ell\}_{\ell\in\N}$ be i.i.d.~non-negative random variables with $\wb F_\varepsilon(x)\in RV_{-\alpha}$, independent from $\{Y_i\}_{i\in\N}$, and set the Karlin model as
\[
\sigma_i:=\varepsilon_{Y_i},
\quad  i\in\N.
\]
A standard way to investigate limit theorems for extremes is to establish the corresponding point-process convergence. 
In \citep{durieu18family}, we 
considered
 for some regularly varying sequence $\{a_n\}_{n\in\N}$ with index $\beta/\alpha$,
\[
\xi_n\topp\sigma:= \summ i1n \ddelta{\sigma_i/a_n,i/n},n\in\N,
\]
and 
proved
 that
\equh\label{eq:DW18}
\xi_n\topp\sigma\weakto \xi_{\alpha,\beta}:=\sif\ell1\summ i1{Q_{\beta,\ell}} \ddelta{\Gamma_\ell^{-1/\alpha},U_{\ell,i}},
\eque
in the space of $\mathfrak M_p((0,\infty]\times[0,1])$ (the space of Radon point measures on $(0,\infty]\times[0,1]$), where in the limit 
 $\{\Gamma_\ell\}_{\ell\in\N}$ are consecutive arrival times of a standard Poisson process, 
 $\{Q_{\beta,\ell}\}_{\ell\in\N}$ are i.i.d.~copies of Sibuya random variables with parameter $\beta$, taking values in $\N$ (see \eqref{eq:Sibuya} below),  $\{U_{\ell,i}\}_{\ell,i\in\N}$ are i.i.d.~uniform random variables from $[0,1]$, and all families are independent.

An interpretation of \eqref{eq:DW18} is as follows. First, at the boundary case $\beta = 1$ ($Q_\beta\weakto 1$ as $\beta\uparrow 1$), the limit in \eqref{eq:DW18}  corresponds to the well-known situation where the limit extremes are {\em independently scattered}, a phenomenon arising from a sequence of i.i.d.~random variables. Namely, in this case
\[
\xi_{\alpha,1} = \sif\ell1\ddelta{\Gamma_\ell^{-1/\alpha},U_\ell},
\]
where $1/\Gamma_\ell^{1/\alpha}$ represents the $\ell$-th largest order statistic and 
$U_{\ell}:=U_{\ell,1}$ its location.
For the general limit point process $\xi_{\alpha,\beta}$ in \eqref{eq:DW18} with $\beta\in(0,1)$, the $\ell$-th order statistic is again represented by $1/\Gamma_\ell^{1/\alpha}$, but it appears at multiple non-local locations (whence the notion of long-range clustering) represented by $\{U_{\ell,i}\}_{i=1,\dots,Q_{\beta,\ell}}$ (notice that  $Q_{\beta,\ell}$ has regularly-varying tail with index $\beta$). 

\subsection{Main result}
We are interested in establishing corresponding limit theorems as in \eqref{eq:DW18} for the 
 Karlin model (memory parameter $\beta\in(0,1)$ and
 regularly-varying
  tail index $\alpha>0$) with multiplicative noise $\{Z_i\}_{i\in\N}$ (non-negative
 with $\wb F_Z(x)\in RV_{-\alpha'}$). It turns out that there are three different regimes in terms of the limit the point process
\[
\xi_n := \summ i1n \ddelta{\varepsilon_{Y_i}Z_i/r_n, i/n},
\]
for some appropriately chosen 
sequence $\{r_n\}_{n\in\N}$, determined by the three parameters.  
For illustration purpose we 
give a simplified statement of the phase transition, assuming Pareto distributions for $Z$ and $\varepsilon$. These assumptions can be significantly relaxed, as proved in Theorem \ref{thm:KSV} later.
Throughout we write $f_n\sim g_n$ as $n\to\infty$ if $\limn f_n/g_n = 1$. 
\begin{theorem}\label{thm:0}
Assume that $\wb F_\varepsilon(x) = x^{-\alpha}, \wb F_Z(x) = x^{-\alpha'}, x\ge 1$, and $p_\ell \sim 
\ell^{-1/\beta}$,
where $\alpha,\alpha'>0$ and $\beta\in(0,1)$. Then, 
\[
\xi_n \weakto \begin{cases}
\sif\ell1 \ddelta{\varepsilon_{Y_\ell}\Gamma_\ell^{-1/\alpha'},U_\ell}, & \text{if }\alpha>\alpha'\beta \mbox{ (noise dominance), with } 
 r_n\sim n^{1/\alpha'},
\\
\\
\sif \ell1\summ i1{Q_{\beta,\ell}}\ddelta{\Gamma_\ell^{-1/\alpha}Z_{\ell,i},U_{\ell,i}}, & \text{if }\alpha<\alpha'\beta\mbox{ (signal dominance), with } r_n\sim c_1n^{\beta/\alpha} , \\\\
\sif\ell1\ddelta{S_\beta^{1/\alpha'}\Gamma_\ell^{-1/\alpha'},U_\ell}  & 
\text{if }\alpha = \alpha'\beta \mbox{ (critical), with }
r_n\sim c_1\pp{n^{\beta}\log (n^\beta)}^{1/\alpha},
\\
\end{cases}
\]
in $\mathfrak M_p((0,\infty]\times[0,1])$, with $c_1 = \Gamma(1-\beta)^{1/\alpha}$, 
$\{\varepsilon_\ell\}_{\ell\in\N}$, $\{Y_\ell\}_{\ell\in\N}$, $\{\Gamma_\ell\}_{\ell\in\N}$, $\{U_\ell\}_{\ell\in\N}$, $\{U_{\ell,i}\}_{\ell\in\N}$, $\{Q_{\beta,\ell}\}_{\ell\in\N}$ independent families defined as before, and $\{Z_{\ell,i}\}_{\ell,i\in\N}$ 
i.i.d.~copies of $Z$ and $S_\beta$ a totally skewed $\beta$-stable random variable, both independent of the preceding families. 
\end{theorem}
We now comment briefly on the three different regimes.
\begin{enumerate}[(i)]
\item The {\em noise-dominance regime} corresponds to the case 
when  the extremes of the perturbed process are caused by the extremes of the multiplicative noise, each multiplied by the corresponding variable from the signal process. 
  This is the easiest case, once one realizes that conditionally on $\vv\varepsilon := \sigma(\{\varepsilon_\ell\}_{\ell\in\N})$, the random variables $\{\varepsilon_{Y_i}\}_{i\in\N}$ are i.i.d. 
Then one immediately sees that, by Breiman's Lemma the above holds under the assumption $\esp(\varepsilon_Y^{\alpha'+\epsilon}\mid\vv\varepsilon)<\infty$ almost surely for some $\epsilon>0$, of which a sufficient condition is $\alpha'\beta<\alpha$.

\item The {\em signal-dominance regime} corresponds to the case when the extremes of the perturbed process are  caused by the extremes of the original signal process, each multiplied by an independent copy of $Z$. 
 Namely, each extreme value of the signal process, say $\Gamma_\ell^{-1/\alpha}$ at location $U_{\ell,i}$, is multiplied by an independent copy $Z_{\ell,i}$. 
To ensure that the limit in this case is Radon (with finite points over $[x,\infty]\times[0,1]$) one needs $\esp(\max_{i=1,\dots,Q_\beta}Z_i^\alpha)<\infty$, a sufficient condition of which is then $\alpha'\beta>\alpha$. 
 Note that with $Z_{\ell,i} \equiv 1$ this becomes the extremal limit theorem in \citep{durieu18family}, and the limit becomes the same as in  \eqref{eq:DW18}.

\item 
In the {\em critical regime},  the limit point process is connected to a known object in extreme value theory (the $\alpha$-Fr\'echet {\em logistic random sup-measures}) which we shall recall later in Remark \ref{rem:logistic}.  
\end{enumerate}
It is remarkable that, while the tail of the limit is determined by the dominant tail of the signal and noise processes,  the memory parameter $\beta$ plays a role in the limit in all three regimes 
(although the influence in the noise-dominance regime is the least and indirectly via the law of $Y$).

We shall also establish the main results in a slightly more general setup with another layer of randomness.
Our model presented here is of  exchangeable nature, hence the phase transition naturally exists for a more general model referred to as the {\em Poisson--Karlin model}, defined on a generic metric space instead of $\R_+$. (The sum-stable counterpart of a variation of Poisson--Karlin model and limit theorems on a generic metric space have been recently investigated in \citep{fu20stable}.)
The Poisson--Karlin model, and essentially the {\em Poissonization technique}, are needed in our earlier work \citep{durieu18family} as in most analysis of the Karlin models since \citep{karlin67central} (see also \citep{gnedin07notes}). 
At the end, applying a new Poissonization method that we developed to the Poisson--Karlin model, we establish  Theorem \ref{thm:0} (and Theorem \ref{thm:KSV}) as a corollary.
 
As for the proofs, the proof for the noise-dominance regime follows from Breiman's Lemma \citep{breiman65some}, where the lighter-tailed random variables are now not i.i.d.~but only so conditionally. The proof for the signal-dominance regime follows from  an adaption of Breiman's Lemma: here, the i.i.d.~random variables with dominant tails are replaced by a sequence of stationary random variables, always with dominant tails. Our adaption follows closely the proof of \citep[Proposition 7.5]{resnick07heavy}, 
and
 it may be of independent interest. 
The most significant contribution of the paper is the limit theorem for the critical regime, where the proof is much more involved. In particular, the phenomenon in this regime cannot be explained by  well-known heuristics related to Breiman's Lemma concerning products of power-law random variables: the top order statistics of both the signal process and the noise actually do not contribute in the limit.

At last, we shall represent the limit theorems in terms of both point-process convergence and 
random-sup-measure
 convergence. The latter is a notion advocated by \citet{obrien90stationary,vervaat97random} for the investigations of extremes of stationary sequences, and in particular when certain dependence structure is preserved in the limit. 
See Section~2.1 for background. 
Working with random sup-measures  helps us to compare the results here with a few other recent developments of extremal limit theorems for {\em long-range clustering} \citep{lacaux16time,samorodnitsky19extremal} where the extremal limit theorems cannot be effectively characterized by point-process convergence alone. 
For the extremes of Karlin model and the one presented here, it is only a matter of convenience for us  whether to state the results in terms of point processes or random sup-measures. We choose to do both.
\subsection{Connections to stochastic volatility model}
There are several models in the form of \eqref{eq:1} in the recent literature, also known as the stochastic volatility models. In this context, $\{\sigma_i\}_{i\in\N}$ is known as the volatility process and $\{Z_i\}_{i\in\N}$ as the innovation process. It is interesting to compare the probabilistic properties of the perturbed Karlin model with those in the literature. At the same time, the perturbed Karlin model may not be the most appropriate to fit financial data.  We conclude the introduction with a couple remarks from these aspects. 
\begin{remark}
For earlier theoretical developments of stochastic volatility models, see \citep[Part II]{andersen09handbook} and in particular the two contributions by Davis and Mikosch. 
Most of the limit theorems regarding stochastic volatility models investigate the case that the extremes are asymptotically independent: the situations are
  nevertheless quite delicate, and could be elaborated further by either modeling the the so-called {\em coefficient of tail dependence} \citep{janssen16stochastic}, or establishing   conditional extreme-value distributions
  for the tail process \citep{kulik15heavy}. 
The easiest way to achieve asymptotic independence is to let the innovation process have the dominant tails. The only two references that we found where the volatility process has the dominant tail are \citet{janssen16stochastic} and \citet{mikosch13stochastic}. 
The former has no extremal clustering as mentioned above, and the second reference demonstrated, via several models, how local clustering may be inherited from the volatility process.  As for volatility processes with long-range dependence, only a few notable references appeared recently; either they belong to the case that the innovation processes have the dominant tails \citep{kulik11tail,kulik13estimation}, or only their functional central limit theorems, no extremes, were studied \citep{kulik12limit}. We are unaware of any models that exhibit the similar long-range clustering or the phase transition as ours. 
\end{remark}

\begin{remark}
The perturbed Karlin model may not be the best to fit financial time series data, and hence we choose not to call it a stochastic volatility model. The empirical evidence of the power-law tails of 
some financial data
 has been well known. 
However, it has been extensively discussed in the literature that for such datasets extremes 
 occur typically in local
clusters (corresponding to extremal index in $(0,1)$)  \citep{basrak02regular}, and yet it can be argued that it is appropriate to apply stochastic volatility models with asymptotic tail-independence (no clustering) for modeling such datasets \citep{drees15statistics,janssen16stochastic}. From this point of view, the feature of long-range clustering 
makes the model not the most appealing.
Nevertheless, in view of the recent result that (a variation of) the Karlin model serves as the counterpart of fractional Brownian motion as the simple random walk to the standard Brownian motion \citep{durieu16infinite}, it is yet to see whether the perturbed Karlin model may find applications in other applied areas. 

Some numerical simulations
for the pertubed Karlin model
 are provided in Figure \ref{fig:1}. 
 Notice that the simulation is misleading about what happens in the critical regime: the simulation seems to suggest that the top statistics of both the signal and the noise may contribute, but in fact neither does in the limit. This phenomenon  would probably require a very large $n$ to be noticeable in simulations.
 Some simulations for the limit logistic random sup-measures at the critical regime are provided in Figure \ref{fig:loRSM}.
\end{remark}
\begin{figure}
\begin{center}
\includegraphics[width = \textwidth]{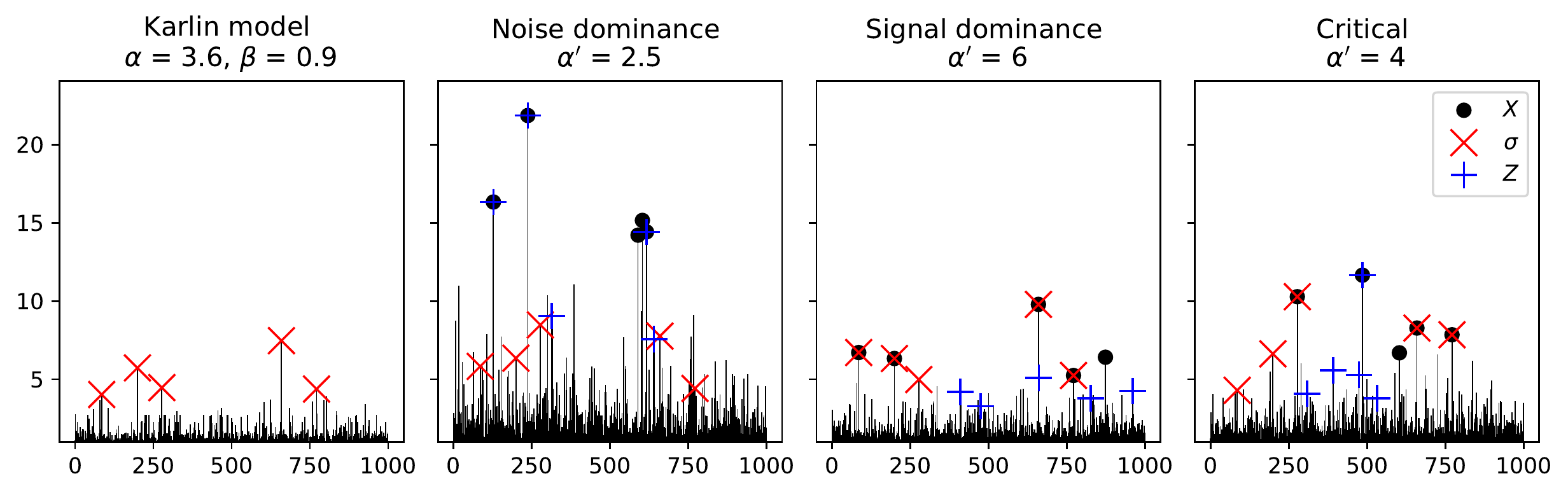}
\end{center}
\caption{Simulations for perturbed Karlin models in three different regimes. 
All simulations are based on the same sampling of Karlin model as the signal process (left plot),
with locations of top 5 values from the model ($X$), the signal process ($\sigma$)
and the noise process ($Z$) marked, respectively. 
}\label{fig:1}
\end{figure}
\begin{figure}[ht!]
\begin{center}
\includegraphics[width = \textwidth]{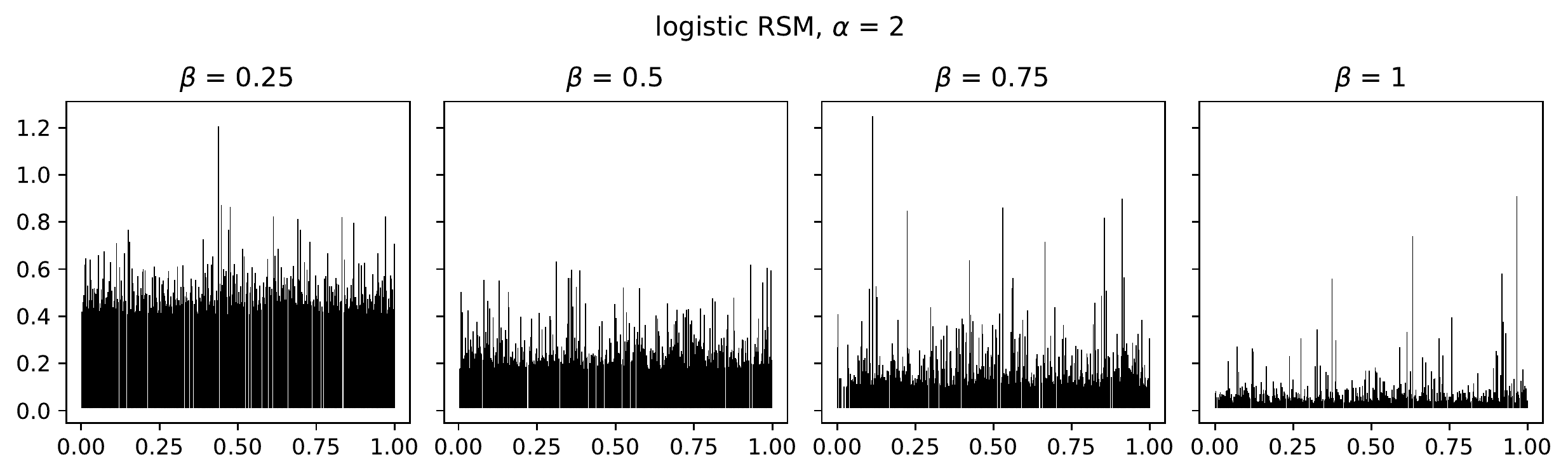}
\end{center} 
\caption{Logistic $\alpha$-Fr\'echet random sup-measures with parameter $\beta\in(0,1)$.  
With $\beta=1$ it becomes $\calM^{\rm is}_\alpha$.}\label{fig:loRSM}
\end{figure}

The paper is organized as follows. In Section \ref{sec:PK} we review the Karlin random sup-measures, the Poisson--Karlin model, and prove an extremal limit theorem for the Poisson--Karlin model. In Section \ref{sec:phase} we state and prove our main theorems regarding the phase-transitions for Poisson--Karlin models with multiplicative noises. In Section \ref{sec:KSV} we prove the corresponding limit theorems for one-dimensional discrete-time model  by a coupling method. 
\NN
\section{Karlin random sup-measures and Poisson--Karlin model}\label{sec:PK}
We review the Karlin random sup-measure and the Poisson--Karlin model, and prove that the empirical random sup-measures of the latter scale to the former, extending our earlier result in \citep{durieu18family}.
\subsection{Background on random sup-measures}
Standard references on random sup-measures and random closed sets are \citep{obrien90stationary,vervaat97random,molchanov16max,molchanov17theory}. 
There is also a recent emerging trend of establishing limit theorems for random sup-measures (e.g.~\citep{lacaux16time,chen20extreme}). 
We only recall a few facts.  

We restrict to sup-measures on a 
metric space $E$
taking values in $[0,\infty]$, and let $\SM(E)$ denote the space of all such sup-measures. 
A sup-mesure $m\in\SM(E)$ 
can be defined as a  
set-function with values in $[0,\infty]$, 
 that is uniquely determined by its evaluations on open subsets of $E$, and that satisfies
\[
 m\left(\bigcup_{x}{G_x}\right)= \sup_x m(G_x)\quad \text{for any collection }\{G_x\}_x\text{ of open subsets of }E.
\]
The canonical extension from open subsets to any set 
$A\subset E$  is given by $m(A)=\inf m(G)$ where the infimum is taken over all open subset $G$ of $E$ such that $A\subset G$.
The space $\SM(E)$ is endowed with the topology of the sup-vague convergence. We say that the sequence $\{m_n\}_{n\in\N}$
 converges sup-vaguely to $m$ as $n\to\infty$ if
\[
 \limsupn m_n(K)\le m(K) \text{ and } \liminf_{n\to\infty} m_n(G)\ge m(G) \text{ for all compact subset $K$ and open subset $G$ of $E$}. 
\]
It is known that under the assumption that $E$ is locally compact and second Hausdorff countable, 
$\SM(E)$ is separable and compact. Now, a random sup-measure is a random element in $\SM(E)$ with the 
$\sigma$-algebra
 of the sup-vague topology.
Since every $m\in\SM(E)$ is uniquely determined by its evaluations on open sets, when identifying random sup-measures in $\SM(E)$ it suffices to restrict to their evaluations on open sets. 
In particular, when comparing two random sup-measures $\calM_1$ and $\calM_2$ on $E$, we shall write
\[
\calM_1(\cdot)= \calM_2(\cdot)
 \quad \mbox{ in short of }\quad \ccbb{\calM_1(A)}_{A\in\calA} =\ccbb{\calM_2(A)}_{A\in\calA}
\]
for some collection $\calA$ 
 of Borel subsets of $E$ 
that form a {\em probability-determining class}, and the equalities above in practice shall be either equalities in the almost-sure sense or equalities for finite-dimensional distributions.

For our limit theorems in the space of random sup-measures, we shall write 
\[
M_n\weakto \calM \quad \mbox{ in short of } \quad\ccbb{M_n(A)}_{A\in\calA}\stackrel{f.d.d.}\weakto\ccbb{\calM(A)}_{A\in\calA}
\]
in $\SM(E)$
for some {\em convergence-determining class} $\calA$ 
of Borel subsets of $E$,
where $\{M_n\}_{n\in\N}$ and $\calM$ are random sup-measures on $E$. Remark that the support of limit random sup-measures $\calM$ in this paper do not have fixed points ($\calM(\{x\}) = 0$ almost surely for all $x\in E$), and in this case the probability-determining and convergence-determining class coincide \citep[Section 12]{vervaat97random}.
Here both of the following are probability/convergence-determining classes \citep[Theorem 12.2]{vervaat97random}: let $D$ be a countable dense set of $E$ and $B(x,r)$ 
denote an open ball in $E$ centered at $x\in E$ with radius $r>0$,
\begin{align*}
\calG_0&:=\ccbb{B(x,r):x\in D, r>0, \wb {B(x,r)}\in\calK},\\
\calK_0&:=\ccbb{\wb {B(x,r)}:x\in D, r>0, \wb {B(x,r)}\in\calK},
\end{align*}
where $\calK \equiv \calK(E)$ is the set of compact subsets of $E$.

Most of our random sup-measures are based on Poisson point processes. By writing
\[
\sum_i \delta_{\xi_i} \sim \PPP(S,\mu),
\]
we mean that $\{\xi_i\}_i$ are measurable enumerations of points from a Poisson point process on $S$ with intensity measure $\mu$. For any real-valued random variable $W$, we write $\wb F_W(x) = \proba(W>x)$. 

\subsection{Karlin random sup-measures}Throughout we fix a locally compact second countable Hausdorff metric space $E$ with
$\calE$ its Borel $\sigma$-algebra,
and a $\sigma$-finite measure $\mu$ on it. Fix $\alpha>0$. A Karlin $\alpha$-Fr\'echet random sup-measure on $(E,\calE)$ with control measure $\mu$ and parameter $\beta\in(0,1]$, denoted by $\mab$ throughout, is a Choquet $\alpha$-Fr\'echet random sup-measure with extremal coefficient functional $\theta(\cdot) = \mu^\beta(\cdot)$
\citep[Definition~3.6 and Theorem~3.7]{molchanov16max}.
As a Choquet $\alpha$-Fr\'echet random sup-measure, the law of $\mab$ is uniquely determined by its marginal law over compact sets 
$K$ of $E$, that is given by 
\[
\proba(\mab(K)\le z) = \exp\pp{-\theta(K)z^{-\alpha}}, 
\; \mfa K\in\calK, z>0.
\]
For limit theorems, it is more convenient to work with series representations.  
Since we are only concerned with the joint law of $\mab$ evaluated at $A_1,\dots,A_d\in\calA :=\{A\in\calE:\mu(A)<\infty\}$ for finite $d\in\N$, 
without loss of generality we assume $\mu(E)<\infty$. The advantage of working under this assumption is to have the simple series representation in \eqref{eq:Sibuya_rep} below.

Throughout we let $Q_\beta$ denote a Sibuya random variable with parameter $\beta\in(0,1)$, which takes values from $\N$ and has probability mass function 
\citep{sibuya79generalized}
\equh\label{eq:Sibuya}
p_k\topp\beta:=
\P(Q_\beta=k)= \frac{\beta\Gamma(k-\beta)}{\Gamma(1-\beta)\Gamma(k+1)},\quad k\in\N.
\eque
Note that $\proba(Q_\beta=k)\sim (\beta/\Gamma(1-\beta))k^{-1-\beta}$ as $k\to\infty$. Equivalently, it is determined by $\esp z^{Q_\beta} = 1-(1-z)^\beta$ for $|z|<1$. 
Introduce
\equh\label{eq:R_E0}
\calR_{\beta} := \bigcup_{i=1}^{Q_\beta}\{U_{i}\},
\eque
where $\{U_{i}\}_{i\in\N}$ are i.i.d.~random element from $E$ with law $\wb\mu:=\mu(\cdot)/\mu(E)$, independent from the Sibuya random variable $Q_\beta$. Let $P_{\beta}$ denote the law of $\calR_{\beta}$ on $\calF_0(E)$, the space of non-empty closed sets of $E$.
Consider
\[
\sif \ell1 \delta_{(\Gamma_\ell,\calR_{\beta,\ell})}\sim \PPP\pp{\R_+\times\calF_0(E),dxdP_{\beta}},
\]
where as a convention $\{\Gamma_\ell\}_{\ell\in\N}$ are ordered in increasing order and $\{\calR_{\beta,\ell}\}_{\ell\in\N}$ can be viewed as i.i.d.~marks. 
\begin{proposition}
Assume that $\mu(E)<\infty$. With the notation above, 
\equh\label{eq:Sibuya_rep}
\calM_{\alpha,\beta}(\cdot)\eqd \mu^{\beta/\alpha}(E)\vif\ell1\frac1{\Gamma_\ell^{1/\alpha}}\inddd{\calR_{\beta,\ell}\cap \cdot\ne\emptyset}, \; \mfa \beta\in(0,1].
\eque
\end{proposition}
\begin{proof}
Let $M$ denote the random sup-measure on the right-hand side of \eqref{eq:Sibuya_rep},
which is a Choquet $\alpha$-Fr\'echet random sup-measure \citep[Theorem 4.4]{molchanov16max}. So it suffices to compute the extremal coefficient functional:
\begin{align*}
\proba\pp{M(K)\le z} & = \exp\pp{-\mu^\beta(E)z^{-\alpha}\proba\pp{\pp{\bigcup_{i=1}^{Q_\beta}\{U_{i}\}}\cap K\ne\emptyset}} = \exp\pp{-\mu^\beta(E)z^{-\alpha}\esp\pp{1-\proba(U_{1}\notin K)^{Q_\beta}}}\\
& = \exp\pp{-\mu^\beta(E)z^{-\alpha}\proba(U_{1}\in K)^\beta} = \exp\pp{-\mu^\beta(K)z^{-\alpha}},
\end{align*}
as desired.
\end{proof}
\begin{remark}With $\beta = 1$, it is well known that the Choquet $\alpha$-Fr\'echet random sup-measure with extremal coefficient functional $\mu(\cdot)$ is 
an independently scattered $\alpha$-Fr\'echet random sup-measure on $E$ with control measure $\mu$ 
\citep[Proposition 6.1]{molchanov16max}: it has a representation
\equh\label{eq:isRSM}
\calM_{\alpha}^{\rm is}\equiv \calM_{\alpha,1}(\cdot) \eqd \vif i1 \xi_i\inddd{U_i\in\cdot} \mbox{ in $\SM(E)$ }  \qmwith \sif i1\delta_{(\xi_i,U_i)}\sim\PPP(\R_+\times E,\alpha x^{-\alpha-1}dx d\mu).
\eque
\end{remark}
\begin{remark}
In the case $\mu(E) = \infty$, the 
representation in \eqref{eq:Sibuya_rep} 
with $\calR_\beta$ as in \eqref{eq:R_E0}
is no longer valid. Another series representation is as follows. We first introduce a $\sigma$-finite measure on  $\calF(E)$, the space of closed sets on $E$. Let $\calN\topp r$ be a Poisson point process on $(E,\calE)$ with intensity measure $r\cdot\mu, r>0$. Then its support, denoted by $\supp\calN\topp r$ (closed by definition) is a random closed set, and hence the law of $\calN\topp r$ induces a probability measure on $\calF(E)$, denoted by $\calL_{\mu, r}$ (determined by $\calL_{\mu,r}(\{F\in\calF(E) : F\cap K\ne\emptyset\}) = 1-e^{-\mu(K)r}$, $K\in\calK(E)$). Then, introduce
\[
\sif i1\delta_{(\xi_i, \calR_{\beta,i})}\sim\PPP(\R_+\times E,\alpha x^{-\alpha-1}dxd\mu_\beta) \qmwith \mu_\beta(\cdot) := \frac1{\Gamma(1-\beta)}\int_0^\infty \beta r^{-\beta-1} \calL_{\mu,r}(\cdot)dr.
\]
We shall also consider 
\[
\sif i1 \delta_{(\xi_i,r_i)} \sim \PPP\pp{\R_+\times \R_+, \alpha x^{-\alpha-1}dx\Gamma(1-\beta)\inv \beta r^{-\beta-1}dr},
\]
and, given the above, conditionally independent Poisson point processes $\{\calN\topp{r_i}_i\}_{i\in\N}$ on $(E,\calE)$ with intensity measure 
$r_i\mu$
 respectively. 
With the notations above, 
\equh\label{eq:KRSM}
\mab(\cdot) \eqd \vif i1 \xi_i\inddd{ \calR_{\beta,i}\cap\cdot\ne\emptyset} \eqd \vif i1 \xi_i\inddd{\supp \calN_i\topp{r_i}\cap \cdot\ne\emptyset}, \; \mfa \beta\in(0,1),
\eque
as Choquet $\alpha$-Fr\'echet random sup-measures on 
$E$.
Indeed, 
the expressions in the middle and on the right-hand side of \eqref{eq:KRSM} are Choquet $\alpha$-Fr\'echet random sup-measures \citep[Theorem 4.4]{molchanov16max}. Therefore it suffices to compute the extremal coefficient functionals. Write
 $\calF_K = \{F\in\calF(E):F\cap K\ne\emptyset\}$. Then,
\begin{align*}
\proba\pp{\vif i1 \xi_i\inddd{ \calR_{\beta,i}\cap K\ne\emptyset}\le z} & = \exp\pp{-z^{-\alpha}\mu_\beta(\calF_K)} = \exp\pp{-z^{-\alpha}\frac1{\Gamma(1-\beta)}\int_0^\infty\beta r^{-\beta-1}\pp{1-e^{-\mu(K)r}}dr}\\
& = \exp\pp{-z^{-\alpha}\mu^\beta(K)},
\end{align*}
where the expression after the second equality is the extremal coefficient functional for the right-hand side of \eqref{eq:KRSM}.
\end{remark}
\subsection{Poisson--Karlin model and its scaling limit}
We introduce the Poisson--Karlin model, of which the special case $(E,\calE) = ([0,1],\calB([0,1]))$
is the poissonized version of the model discussed in introduction.
We shall then prove that the empirical random sup-measures of the Poisson--Karlin model converge in distribution to the Karlin random sup-measure.

From now on, we restrict ourselves to the case that
\[
\mu(E) = 1,
\]
and we have seen in this case, 
\equh\label{eq:KRSM0}
\mab(\cdot) \eqd \sup_{\ell\in\N}\frac1{\Gamma_\ell^{1/\alpha}}\inddd{ \calR_{\beta,\ell}\cap\cdot\ne\emptyset}\qmwith \sif \ell1\ddelta{\Gamma_\ell, \calR_{\beta,\ell}}\sim \PPP(\R_+\times \calF_0(E),dxdP_{\calR_\beta}).
\eque

Now for the Poisson--Karlin model, introduce the following families of random variables, and assume all four families are independent.
\begin{itemize}
\item Let $\{U_i\}_{i\in\N}$ be i.i.d.~random elements from $E$ with law $\mu$. 
\item Let $\{\varepsilon_\ell\}_{\ell\in\N}$ be i.i.d.~non-negative random variables satisfying $\wb F_\varepsilon(x)\in RV_{-\alpha}$.
\item Let $\{Y_i\}_{i\in\N}$ be i.i.d.~$\N$-valued random variables so that, with $p_k:= \proba(Y_1 = k), k\in\N$ being non-increasing and $\nu:=\sum_{k\in\N}\delta_{1/p_k}$ satisfying
\equh\label{eq:nu}
 \nu(x):=\nu((0,x])=\max\ccbb{k\in\N: \frac1{p_k} \le x} = x^{\beta}L(x),\quad x>0,\, \beta\in(0,1),
\eque
for a slowly varying function $L$ at infinity. 
\item Let $N(\lambda)$ be another Poisson random variable with mean $\lambda>0$. 
\end{itemize}
Then, by the Poisson--Karlin model we refer to the following point process
\[
 \summ i1{N(\lambda)}\ddelta{\varepsilon_{Y_i},U_i},
\]
and in this paper we are interested in its empirical random sup-measure defined as
\[
M_\lambda(\cdot):=\max_{i:U_{i}\in \cdot}\varepsilon_{Y_i}.
\] 
When considering limit theorems, without loss of generality we examine only $\lambda=n\in\N$. 
The following result generalizes the main result of \citep{durieu18family}. (See \citep{fu20stable} for how the Poisson--Karlin model leads to {\em sum-stable} random fields.)
We let $\mathfrak M_p(S)$ denote the space of Radon point measures on a topological space $S$, 
and
$\calF_0(E)$ is the space of non-empty compact subsets of $E$ equipped with Fell topology.

\begin{theorem}\label{thm:DW18}
Consider the Poisson--Karlin model. Assume that $\wb F_\varepsilon(x)\in RV_{-\alpha}$ with $\alpha>0$, and $\nu$ satisfies \eqref{eq:nu}. Consider 
\[
R_{n,\ell}:=\bigcup_{i=1,\dots,N(n):Y_i = \ell} \{U_{i}\}, \quad \ell \in\N.
\]
Then for $\{a_n\}_{n\in\N}$ satisfying
\[
\limn\Gamma(1-\beta)\nu(n)\wb F_\varepsilon(a_n)=1,
\]
we have (with notations from \eqref{eq:KRSM0})
\equh\label{eq:DW18PP}
\sif\ell1\ddelta{\varepsilon_\ell/a_n,R_{n,\ell}} \weakto \sif \ell1\ddelta{\Gamma_\ell^{-1/\alpha}, \calR_{\beta,\ell}},
\eque
as $n\to\infty$ in $\mathfrak M_p((0,\infty]\times\calF_0(E))$,
and as a consequence, 
\[
\frac1{a_n} M_n\weakto \calM_{\alpha,\beta}
\]
as $n\to\infty$ in $\SM(E)$. 
\end{theorem}
\begin{remark}
It is known that the convergence of the point processes \eqref{eq:DW18PP} implies the convergence of the corresponding random sup-measures. See \citep[Theorem 4.2]{durieu18family}. So for all our limit theorems we only establish the point-process convergence. 
\end{remark}
\begin{proof}[Sketch of the proof]
The proof is essentially the same as in \citep[Theorem 4.1]{durieu18family}, where only
the case $(E,\calE) = ([0,1],\calB([0,1]))$ was considered. We only sketch the key steps shedding light on how the Sibuya distribution appears in the limit. 
We first 
introduce the following statistics:
\equh\label{eq:statistics}
K_{n,\ell}:= \summ i1{N(n)}\inddd{Y_i = \ell}, \quad K_n :=\sif \ell1 \inddd{K_{n,\ell}>0}, \quad
J_{n,k}:=\sif \ell1 \inddd{K_{n,\ell} = k}.
\eque
Note that the left-hand side of \eqref{eq:DW18PP} is restricted to $\mathfrak M_p((0,\infty]\times\calF_0(E))$, so the points corresponding to those $\ell$ such that $R_{n,\ell} = \emptyset$ are not involved. Let $\what L_n$ denote the collection of all 
other
 $\ell$. So $|\what L_n| = K_n$.
 Then we rewrite the left-hand side of \eqref{eq:DW18PP} as
\[
\sum_{\ell\in\what L_n}\ddelta{\varepsilon_\ell/a_n,R_{n,\ell}}.
\]
Next, we order $\{\varepsilon_\ell\}_{\ell\in\what L_n}$ in decreasing order, and assume that there are no ties for the sake of simplicity. Let $\{\what \ell_{n,1},\dots,
\what \ell_{n,K_n}\}= \what L_n$ be the corresponding relabellings such that the reordering is
\[
\varepsilon_{\what \ell_{n,1}}>\cdots>\varepsilon_{\what \ell_{n,
K_n}}.
\]
It is a standard argument to focus first on say the top $r$ largest $\varepsilon$, and then let $r\to\infty$ eventually. We only elaborate the first part here, and fix $r\in\N$. The goal is then to show that
\[
\summ i1r \ddelta{\varepsilon_{\what \ell_{n,i}}/a_n,R_{n,\what\ell_{n,i}}}\weakto \summ i1r \ddelta{\Gamma_\ell^{-1/\alpha}, \calR_{\beta,\ell}}.
\]
To see the above holds, we first recall that  \citep{karlin67central,gnedin07notes}
\[
\limn \frac{K_n}{\nu(n)} = \Gamma(1-\beta) \mbox{ almost surely.}
\]
So, since $K_n\wb F_\varepsilon(a_n)\sim1$ almost surely, we have
\[
\frac1{a_n}\pp{\varepsilon_{\what\ell_{n,1}},\dots,\varepsilon_{\what \ell_{n,r}}} \weakto \pp{\Gamma_1^{-1/\alpha},\dots,\Gamma_r^{-1/\alpha}},
\] following from a well-known fact in extreme-value theory for i.i.d.~random variables with power-law tails \citep{resnick07heavy}, and it remains to show
\[
\pp{R_{n,\what \ell_{n,1}},\dots,R_{n,\what\ell_{n,r}}}\weakto (\calR_{\beta,1},\dots,\calR_{\beta,r}).
\]
In view of the representation of $R_{n,\ell}$ and $ \calR_{\beta,\ell}$, it suffices to prove that, for $\what Q_{n,s} := |R_{n,\what \ell_{n,s}}|$, 
\[
\pp{\what Q_{n,1},\dots,\what Q_{n,r}}\weakto (Q_{\beta,1},\dots,Q_{\beta,r}).
\]
Since $K_{n,\ell} = |R_{n,\ell}|$, the left-hand side corresponds to the law of sampling without replacement of $r$ elements from $K_n$ elements $\{K_{n,\ell}\}_{\ell\in\N:K_{n,\ell}>0}$, consisting of $J_{n,k}$ of $k$ 
for each 
$k\in\N$. So we have, for $r$ fixed,
\[
\proba\pp{\what Q_{n,j} = k\mmid \calY} = \frac{J_{n,k}}{K_n} \mbox{ almost surely,} \quad k\in\N,\, j=1,\dots,r,
\]
with $\calY = \sigma(\{Y_i\}_{i\in\N})$. 
Moreover, it is easy to show that $\what Q_{n,1},\dots,\what Q_{n,r}$ are asymptotically independent.
Therefore, it remains to show that $\what Q_{n,1}\weakto Q_\beta$, which is a well known fact for the Karlin model (a.k.a.~the paintbox random partition \citep{pitman06combinatorial}).
\end{proof}

\NN
\section{Phase transitions for Poisson--Karlin model with multiplicative noise}\label{sec:phase}
We introduce multiplicative noise to the Poisson--Karlin model. Let $\{Z_i\}_{i\in\N}$ be 
non-negative i.i.d.~random variables. 
 Assume furthermore that $\{Z_i\}_{i\in\N}$ are independent from the Poisson--Karlin model introduced above with parameters $\alpha>0,\beta\in(0,1)$. 
We are interested in the 
point process
\[
\summ i1{N(n)}\ddelta{\varepsilon_{Y_i}Z_i/r_n,U_i},
\]
for some appropriately chosen sequence $\{r_n\}_{n\in\N}$, and
 random sup-measure $M_n$ defined by
\[
M_n(\cdot):=\sup_{i=1,\dots,N(n):U_{i}\in \cdot} \varepsilon_{Y_i}Z_i.
\]
For a brief overview,  we assume that $\wb F_Z(x)\in RV_{-\alpha'}$ for some $\alpha'>0$, and this condition might be relaxed or strengthened later.
There are three different regimes for the scaling limits of $M_n$ depending on the relation between $\alpha$ and $\beta\alpha'$:
\begin{enumerate}[(i)]
\item Noise-dominance regime: $\alpha>\beta\alpha'$, 
Section \ref{sec:nd},
\item Signal-dominance regime: $\alpha<\beta\alpha'$, 
Section \ref{sec:sd},
\item Critical regime, $\alpha = \beta\alpha'$ 
Section \ref{sec:critical}.
\end{enumerate}

Throughout, we let $C$ denote a 
positive constant that may change from line to line. 

\subsection{Noise-dominance regime}
\label{sec:nd}
The main theorem in this regime is the following. 

\begin{theorem}\label{thm:nd}
Assume that $\wb F_Z(x)\in RV_{-\alpha'}$ and $\esp_{\vv\varepsilon} \varepsilon_Y^{\alpha'+\epsilon}<\infty$ 
almost surely
 for some $\epsilon>0$. 
Then, for any sequence $\{c_n\}_{n\in\N}$ 
such that
\[
\limn n\wb F_Z(c_n) =1,
\]
conditionally on $\vv\varepsilon$,
\[
\sum_{i=1}^{N(n)} \ddelta{\varepsilon_{Y_i}Z_i/c_n,U_i}\weakto
\sif\ell1\ddelta{\varepsilon_{Y_\ell}\Gamma_\ell^{-1/\alpha'}, U_{\ell}}
\]
in $\mathfrak M_p((0,\infty]\times E)$ almost surely. As a consequence, conditionally on $\vv\varepsilon$,
\equh\label{eq:as_weak_convergence}
\frac{1}{c_n}M_n \weakto \pp{\esp_{\vv\varepsilon}\varepsilon_Y^{\alpha'}}^{1/\alpha'}\cdot\cM^{{\rm is}}_{\alpha'} 
\eque
in $\SM(E)$ almost surely.
\end{theorem}
The convergence in \eqref{eq:as_weak_convergence} is understood as the almost-sure weak convergence with respect to $\calE$. That is, 
\[
\limn \esp_{\vv\varepsilon} f\pp{c_n\inv M_n} = \esp_{\vv\varepsilon} f\pp{\pp{\esp _{\vv\varepsilon}\varepsilon_Y^{\alpha'}}^{1/\alpha'}\cdot \calM_{\alpha'}^{\rm is}} \mbox{ almost surely,}
\] 
 for all continuous and bounded 
 functions $f:\SM(E)\to \R$. The corresponding point-process convergence is interpreted similarly.

Before proving the limit theorem, we first examine the limit random sup-measure.
\begin{lemma}
Assume that $\esp_{\vv\varepsilon}\varepsilon_Y^{\alpha'}<\infty$ almost surely. Then, 
\[
(\esp_{\vv\varepsilon}\varepsilon_Y^{\alpha'})^{1/\alpha'}\cdot\calM_{\alpha'}^{\rm is}(\cdot) \eqd \sup_{i\in\N}\frac1{\Gamma_i^{1/\alpha'}}\varepsilon_{Y_i}\inddd{U_i\in\cdot}.
\]
\end{lemma}

\begin{proof}
Indeed, given $\vv\varepsilon$, $\{\varepsilon_{Y_i}\}_{i\in\N}$ are i.i.d.~random variables, and the above follows from
 \equh\label{eq:iidmarks}
 \sif \ell1 \delta_{\Gamma_i^{-1/\alpha'}\varepsilon_{Y_i}} \eqd \sif i1 \delta_{(\esp_{\vv\varepsilon} \varepsilon_Y^{\alpha'})^{1/\alpha'}\Gamma_i^{-1/\alpha'}}\; \mbox{ almost surely.}
 \eque
Conditionally on $\vv\varepsilon$, the left-hand side is again a Poisson point process \citep[Proposition 5.2]{resnick07heavy}, and it suffices to compute the intensity measure evaluated at the region $(z,\infty)$, which equals
\begin{align*}
&\int_0^\infty\int_0^\infty \inddd{xy^{-1/\alpha'}>z} dF_{\varepsilon_Y\mid \vv\varepsilon}(x) dy = z^{-\alpha'}\int_0^\infty x^{\alpha'} dF_{\varepsilon_Y\mid \vv\varepsilon}(x) = z^{-\alpha'}\esp_{\vv\varepsilon}\varepsilon_Y^{\alpha'}\; \mbox{ almost surely.}
\end{align*}
\end{proof}
\begin{proof}[Proof of Theorem \ref{thm:nd}]
We shall then work with the representation of the limit random sup-measure based on the left-hand side of \eqref{eq:iidmarks}.  It suffices to prove the convergence of point processes.   Since
$\wb F_Z(x)\in RV_{-\alpha'}$, it follows that \citep[Theorem 5.3]{resnick07heavy}
\[
 \summ i1{N(n)}\delta_{Z_i/c_n}\weakto \sif i1 \delta_{\Gamma_i^{-1/\alpha'}},
\]
whence, conditioning on $\vv\varepsilon$, 
\equh\label{eq:Sid}
\summ i1{N(n)}\ddelta{Z_i/c_n,\varepsilon_{Y_i},U_i}\weakto \sif i1 \ddelta{\Gamma_i^{-1/\alpha'},\varepsilon_{Y_i},U_i} \eque
in $\mathfrak M_p((0,\infty]\times(0,\infty)\times E)$, almost surely.
The third coordinates can be viewed as i.i.d.~marks and do not change in the limiting procedure, and hence can be omitted in the analysis. The goal is then to show that \eqref{eq:Sid} implies
\equh\label{eq:Sid2}
\summ i1{N(n)}\ddelta{Z_i\varepsilon_{Y_i}/c_n}\weakto \sif i1 \ddelta{\Gamma_i^{-1/\alpha'}\varepsilon_{Y_i}}
\eque 
as $n\to\infty$ in $\mathfrak M_p((0,\infty])$, almost surely.
Here $N(n)$ is Poisson with parameter $n$. Remark that  if one replaces $N(n)$ by $n$ above, \citep[Proposition 7.5]{resnick07heavy} proves exactly that \eqref{eq:Sid} implies \eqref{eq:Sid2}, provided $\esp_{\vv\varepsilon}\varepsilon_Y^{\alpha'+\epsilon}<\infty$ almost surely for some $\epsilon>0$. Since $N(n)$ is independent from the other random variables, the analysis here is essentially the same. We omit the details. \end{proof}

We conclude this section by elaborating on the conditions $\esp_{\vv\varepsilon}\varepsilon_Y^{\alpha'+\epsilon}<\infty,\epsilon\ge 0$. Note that in our limit theorem we need $\epsilon>0$, while for the limit random sup-measure to be finite almost surely, $\epsilon=0$ is sufficient (and this condition is also necessary). 
We say a function $f$ is dominated by a function $g\in RV_\gamma$ at infinity, if for all $x$ large enough, $f(x)\le Cg(x)$.
\begin{lemma}
For $\alpha>\alpha'\beta$, assume the following 
assumptions:
\begin{enumerate}[(i)]
\item  $\wb F_\varepsilon(x)$ is dominated by a function in $RV_{-\alpha}$ at infinity,
\item $\nu(x)$ (recall \eqref{eq:nu}) is dominated by a function in $RV_{\beta}$ at infinity.
\end{enumerate}
Then,
$\esp_{\vv\varepsilon} \varepsilon_Y^{\alpha'+\epsilon}<\infty$ almost surely for all $\epsilon\in[0,\alpha/\beta-\alpha')$. 
\end{lemma}
\begin{proof}
By definition, $\E_{\vv\varepsilon}\varepsilon_Y^{\alpha'+\epsilon}=\sif \ell1 p_\ell\varepsilon_\ell^{\alpha'+\epsilon}$. The convergence of this series follows from the Kolmogorov's three-series theorem. Indeed, first for any $c>0$,
\[
\sif \ell1\P(p_\ell\varepsilon_\ell^{\alpha'+\epsilon}>c)= \sif\ell1 
 \wb F_{\varepsilon}((c/p_\ell)^{1/(\alpha'+\epsilon)})\le C\sif \ell1 p_\ell^{\alpha/(\alpha'+\epsilon)-\epsilon_1}
\]
for some small $\epsilon_1>0$ by Potter's bound. 
Assume that \eqref{eq:nu} holds, which is equivalent to  that $p_\ell\in RV_{-1/\beta}$ as $\ell\to\infty$, and hence the above is bounded by $C\sif \ell1\ell^{-(1/\beta-\epsilon_1)(\alpha/(\alpha'+\epsilon)-\epsilon_1)}$ by Potter's bound again. By the assumption $\alpha'+\epsilon<\alpha/\beta$, one can tune $\epsilon_1>0$ small enough so that the power over $\ell$ is strictly less than $-1$, and hence the series is finite. 
Next, choose $\beta'\in(\beta,1\wedge \alpha/(\alpha'+\epsilon))$. Then,  $\sum_{\ell=1}^\infty p_\ell^{\beta'}<\infty$ as $\beta'>\beta$ and $\E\varepsilon^{\alpha'\beta'}<\infty$ as $(\alpha'+\epsilon)\beta'<\alpha$. It then follows that
\[
 \sif\ell1\E\left(p_\ell \varepsilon_\ell^{\alpha'+\epsilon} \ind_{\{p_\ell\varepsilon_\ell^{\alpha'+\epsilon}\le c \}}\right) \le c^{1-\beta'}\sif \ell1 p_\ell^{\beta'}\esp \varepsilon^{(\alpha'+\epsilon)\beta'}<\infty
\]
and 
$\sif\ell1\Var(p_\ell \varepsilon_\ell^{\alpha'+\epsilon} \inddd{p_\ell\varepsilon_\ell^{\alpha'}\le c})<\infty$, 
where $\Var$ stands for the variance. The proof can be modified to prove the case that $\nu(x)$ is dominated by a function in $RV_\beta$. 
\end{proof}

\begin{remark}
Assume that $\nu$ satisfies \eqref{eq:nu}, $\wb F_\varepsilon(x)\in RV_{-\alpha}$ and  $\wb F_Z(x)\in RV_{-\alpha'}$. The above says that if $\alpha>\alpha'\beta$ then $\esp_{\vv\varepsilon}\varepsilon_Y^{\alpha'}<\infty$ almost surely. For this to hold at the boundary case when $\alpha = \alpha'\beta$, a necessary and sufficient condition is that
\equh\label{eq:nd_well_defined}
\esp\nu\pp{\varepsilon^{\alpha'}}<\infty.
\eque
In particular when $\nu(n)\sim Cn^\beta$, the above is equivalent to $\esp\varepsilon^{\alpha'\beta} <\infty$. To see this, apply the three-series theorem to  $\esp_{\vv\varepsilon} \varepsilon_Y^{\alpha'}=\sif\ell1p_\ell\varepsilon_\ell^{\alpha'}$. 
The first series becomes
\equh\label{eq:3series}
 \sif\ell1\P\pp{p_\ell\varepsilon_\ell^{\alpha'}>1} = \int_0^\infty \wb F_\varepsilon(x^{1/\alpha'})\nu(dx)=\int_0^\infty \wb F_\varepsilon(y)\nu(d(y^{\alpha'})).
\eque
Note that, by integration by parts, 
\equh\label{eq:=infty}
\int_0^a \wb F_\varepsilon(y)\nu(d(y^{\alpha'}))=\wb F_\varepsilon(a)\nu(a^{\alpha'}) + \int_0^a \nu\spp{y^{\alpha'}}d F_\varepsilon(y).
\eque
Observe also that $0\le \wb F_\varepsilon(a)\nu(a^{\alpha'})\le \esp \nu(\varepsilon^{\alpha'})$ for $a>0$ and that $\int_0^a \nu\spp{y^{\alpha'}}d F_\varepsilon(y)\to \esp \nu\spp{\varepsilon^{\alpha'}}$ as $a\to\infty$. Then, \eqref{eq:3series} is finite, if and only if \eqref{eq:nd_well_defined} holds.
For the second series, applying 
$\esp \spp{\varepsilon^{\alpha'}\inddd{\varepsilon\le x}}\sim (\alpha'-\alpha)\inv x^{\alpha'}
\wb F_\varepsilon(x)
$ as $x\to\infty$ with $\alpha'>\alpha$,
we have that
\[
\sif \ell1 \esp\spp{p_\ell\varepsilon_\ell^{\alpha'}\indd{p_\ell\varepsilon_\ell^{\alpha'}\le 1}}\le C\sif\ell1\wb F_\varepsilon(p_\ell^{-1/\alpha'}) = C\int_0^\infty \wb F_\varepsilon(x^{1/\alpha'})\nu(dx),
\]
the same upper bound as in \eqref{eq:3series}. The third series can be treated similarly.
Note that \eqref{eq:=infty} also says that if $\alpha<\alpha'\beta$, then $\esp_{\vv\varepsilon}\varepsilon_Y^{\alpha'}=\infty$ almost surely.
\end{remark}
\subsection{Signal-dominance regime}
\label{sec:sd}
Throughout we write 
\equh\label{eq:ZW}
\wt Z_W \equiv \max_{i=1,\dots,W}Z_i,
\eque where $W$ is an $\N$-valued random variable (possibly a constant) that is {\em assumed to be independent from $\{Z_i\}_{i\in\N}$}. 
The main theorem of this regime is the following.

\begin{theorem}\label{thm:sd}
Assume that $\wb F_\varepsilon(x)\in RV_{-\alpha}$, and that 
\equh\label{eq:cond_sd}
\esp \wt Z_{Q_{\beta-\epsilon}}^{\alpha+\epsilon}<\infty \mbox{ for some } \epsilon>0.
\eque
For any sequence $\{a_n\}_{n\in\N}$ such that 
\equh\label{eq:an}
\limn\Gamma(1-\beta)\nu(n)\wb F_\varepsilon(a_n)= 1,
\eque
we have
\equh\label{eq:PP_convergence_sd}
\summ i1{N(n)}\ddelta{\varepsilon_{Y_i}Z_i/a_n,U_i}\weakto \sif \ell1\summ i1{Q_{\beta,\ell}}\ddelta{\Gamma_\ell^{-1/\alpha}Z_{\ell,i},U_{\ell,i}}
\eque
as $n\to\infty$ in $\mathfrak M_p((0,\infty]\times E)$, 
where 
$\{Z_{\ell,i}\}_{\ell,i\in\N}$ i.i.d.\ with law $F_Z$ and $\{U_{\ell,i}\}_{\ell,i\in\N}$ i.i.d.\ with law $\mu$, both independent from the other families of random variables.
As a consequence,
\[
\frac{1}{a_n}M_n(\cdot) \weakto \mabZ(\cdot) := \vif \ell1 \frac1{\Gamma_{\ell}^{1/\alpha}}\max_{i=1,\dots,Q_{\beta,\ell}}Z_{\ell,i}\inddd{U_{\ell,i}\in\cdot} 
\]
as $n\to\infty$ in ${\rm SM}(E)$.

\end{theorem}
Notice that it is straightforward to see that for $\mabZ$ to be almost surely finite, a sufficient and necessary condition is $\esp\wt Z_{Q_\beta}^\alpha<\infty$. Indeed, 
for every open set $A\subset E$, 
writing $\wt Z_{Q_\beta}(A):=\max_{i=1,\dots,Q_\beta: U_i\in A}Z_i$, 
\begin{align*}
\P\left(\mabZ(A)\le z\right)
&=\P\left(\sup_{\ell\ge 1}\Gamma_\ell^{-1/\alpha} \max_{i=1,\dots,Q_{\beta,\ell}:U_{\ell,i}\in A}Z_{\ell,i} \le z \right)\\
&=\exp\left(-\int_0^\infty\wb F_{\wt Z_{Q_\beta}(A)}(z/x)\alpha x^{-\alpha-1}dx \right) =\exp\left(- z^{-\alpha} \E\wt Z_{Q_\beta}^\alpha(A)\right).
\end{align*}
Again, the condition for $\mabZ$ to be finite almost surely (\eqref{eq:cond_sd} with $\epsilon=0$) is strictly weaker than what is needed for the convergence. To see that \eqref{eq:cond_sd} holds for $\wb F_Z\in RV_{-\alpha'}$ with $\alpha<\alpha'\beta$, it suffices to pick $\epsilon>0$ such that 
$\alpha+\epsilon<(\beta-\epsilon)\alpha'$. 
Indeed, 
\equh\label{eq:interpretation}
\esp \wt Z^{\alpha+\epsilon}_{Q_{\beta-\epsilon}} =C\int_0^\infty x^{\alpha+\epsilon-1}\wb F_{\wt Z_{Q_{\beta-\epsilon}}}(x)dx \le C\pp{1+\int_1^\infty x^{\alpha+\epsilon-1}x^{-(\beta-\epsilon)\alpha'}L_Z^{\beta-\epsilon}(x)dx}
\eque
 for some slowly-varying function $L_Z$, where in the last step we used $\wb F_{\wt Z_{Q_\beta}}(x) = \wb F_Z(x)^\beta$.
\begin{remark}
In view of \eqref{eq:interpretation}, the condition $\esp \wt Z_{Q_{\beta-\epsilon}}^{\alpha+\epsilon}<\infty$ 
is slightly more 
restrictive
 than $\esp \wt Z_{Q_\beta}^\alpha<\infty$. This is similar in spirit to the condition in Breiman's Lemma: for non-negative independent random variables $X,Y$, $\wb F_Y(x)\in RV_{-\alpha}$,  for the limit theorem $\lim_{x\to\infty}\wb F_{XY}(x)/\wb F_Y(x)= \esp X^\alpha$ to hold, one needs $\esp X^{\alpha+\epsilon}<\infty$ for some $\epsilon>0$. 
\end{remark}
\begin{proof}[Proof of Theorem \ref{thm:sd}]
We focus on \eqref{eq:PP_convergence_sd}. 
Recall $K_{n,\ell}$ in \eqref{eq:statistics}. 
We have seen in Theorem \ref{thm:DW18} that
\[
\summ i1{N(n)}\ddelta{\varepsilon_{Y_i}/a_n,U_i}\eqd \sif\ell1\summ i1{K_{n,\ell}}\ddelta{\varepsilon_\ell/a_n,U_{\ell,i}}\weakto \sif \ell1 \summ i1{Q_{\beta,\ell}}\ddelta{\Gamma_\ell^{-1/\alpha},U_{\ell,i}},
\]
whence
\equh\label{eq:joint_sd}
\summ i1{N(n)}\ddelta{\varepsilon_{Y_i}/a_n,Z_i,U_i}\eqd \sif\ell1\summ i1{K_{n,\ell}}\ddelta{\varepsilon_\ell/a_n,Z_{\ell,i},U_{\ell,i}}\weakto \sif \ell1 \summ i1{Q_{\beta,\ell}}\ddelta{\Gamma_\ell^{-1/\alpha},Z_{\ell,i},U_{\ell,i}}.
\eque
The third coordinates of the points can be viewed as i.i.d.~marks and they do not change in the limit. So it suffices to focus on
\[
 \eta_n:=\summ i1{N(n)}\delta_{\varepsilon_{Y_i}Z_{i}/a_n}\qmand  \eta:=\sif\ell1\sum_{i=1}^{Q_{\beta,\ell}} \delta_{\Gamma_\ell^{-1/\alpha}Z_{\ell,i}},
\]
and prove
\[
 \eta_n \weakto \eta \quad \text{ in }\mathfrak M_p((0,\infty]).
\]
Note that we cannot directly apply the product functional to \eqref{eq:joint_sd} as $\{(x,y)\in (0,\infty]\times(0,\infty) : |xy|\ge 1\}$ is not compact in $(0,\infty]\times(0,\infty)$.
 The proof follows the approach of Resnick \cite[Proposition~7.5]{resnick07heavy}.
Let $\delta \in (0,1)$ and
\[
 \Lambda_\delta 
 :=
  \ccbb{(x,y)\in (0,\infty]\times(0,\infty) :x\ge \delta, y\in \bb{\delta, \delta^{-2}}}.
\]
It is a compact subset of $(0,\infty]\times(0,\infty)$ and by restriction,
\[
\summ i1{N(n)}\ddelta{\varepsilon_{Y_i}/a_n,Z_i}(\Lambda_\delta\cap\cdot)\weakto  \sif \ell1 \summ i1{Q_{\beta,\ell}}\ddelta{\Gamma_\ell^{-1/\alpha},Z_{\ell,i}}(\Lambda_\delta\cap\cdot)\quad\text{ in }\mathfrak M_p(\Lambda_\delta).
\]
Since for any $c>0$, $\{(x,y)\in \Lambda_\delta : |xy|\ge c\}$ is a compact subset of $\Lambda_\delta$, we can use the product functional to get
\[
 \eta_{n,\delta}:=\sum_{i=1}^{N(n)} \delta_{\varepsilon_{Y_i}Z_i/a_n}\inddd{(\varepsilon_{Y_i}/a_n, Z_i)\in\Lambda_\delta}\weakto \eta_\delta:=\sif \ell1\sum_{i=1}^{Q_{\beta,\ell}} \delta_{\Gamma_\ell^{-1/\alpha}Z_{\ell,i}}\inddd{(\Gamma_\ell^{-1/\alpha},Z_{\ell,i})\in\Lambda_\delta},
\]
as $n\to\infty$ in $\mathfrak M_p((0,\infty])$.
Further, $\eta_\delta\weakto \eta$, as $\delta\downarrow 0$. To conclude, it remains to prove that for all positive continuous 
functions
 $f$ with compact support in $(0,\infty]$ and all $\epsilon>0$,
\[
 \lim_{\delta\downarrow0}\limsupn \P\left( \sum_{i=1}^{N(n)} f\left(\varepsilon_{Y_i}Z_i/a_n\right)\inddd{(\varepsilon_{Y_i}/a_n,Z_i)\notin\Lambda_\delta}>\epsilon\right)=0.
\]
Fix such a function $f$ and a real $\kappa>0$ such that $f\equiv 0$ on $(0,\kappa]$. It is sufficient to prove that
\begin{equation}\label{eq:Lambda_approx}
  \lim_{\delta\downarrow0}\limsupn\P\left(\bigcup_{i=1}^{N(n)} \ccbb{\varepsilon_{Y_i}Z_i/a_n>\kappa,\, (\varepsilon_{Y_i}/a_n,Z_i)\in\Lambda_\delta^c}\right)=0,
\end{equation}
where $\Lambda_\delta^c:=((0,\infty]\times(0,\infty))\setminus\Lambda_\delta$. The proof of \eqref{eq:Lambda_approx} is divided into 4 steps by writing $\Lambda_\delta^c$ as the disjoint union of the sets 
\begin{align*}
A_{1,\delta}&:=\left\{(x,y)\in(0,\infty]\times(0,\infty) : x<\delta,\,y<\delta^{-1/2} \right\},\\
A_{2,\delta}&:=\left\{(x,y)\in(0,\infty]\times(0,\infty) : x<\delta,\,y\ge\delta^{-1/2}  \right\},\\
A_{3,\delta}&:=\left\{(x,y)\in(0,\infty]\times(0,\infty) : x\ge\delta,\,y<\delta \right\},\\
A_{4,\delta}&:=\left\{(x,y)\in(0,\infty]\times(0,\infty) : x\ge \delta,\, y>\delta^{-2}\right\}.\\
\end{align*}
Write 
\[
 E_{j,\delta,n}:=\P\left(\bigcup_{i=1}^{N(n)} \left\{\varepsilon_{Y_i}Z_i/a_n>\kappa,\, (\varepsilon_{Y_i}/a_n,Z_i)\in A_{j,\delta} \right\}\right).
\]

\medskip
\noindent 1)
If $(\varepsilon_{Y_i}/a_n,Z_i)\in A_{1,\delta}$, then $\varepsilon_{Y_i}Z_i/a_n<\delta^{1/2}$. Thus, when $\delta^{1/2}\le \kappa$, $E_{1,\delta,n} = 0$ for all $n\in\N$. 
 \medskip
 
\noindent 2) 
 Let $\calY = \sigma(\{Y_i\}_{i\in\N})$, with respect to which $K_{n,\ell}$ is measurable.  We start by writing that
\begin{align*}
 E_{2,\delta,n}
&\le \P\left(\bigcup_{i=1}^{N(n)} \left\{\varepsilon_{Y_i}Z_i/a_n>\kappa,\, Z_i>\delta^{-1/2} \right\}\right)\le \E\left( \sum_{\ell= 1}^\infty \P\left(\varepsilon_{\ell}\wt Z_{K_{n,\ell}}/a_n>\kappa,\, \wt Z_{K_{n,\ell}}>\delta^{-1/2}\mmid \calY,N(n)\right)\right)\\
&=\sif k1 \E J_{n,k}\P\left(\varepsilon\wt Z_{k}>\kappa a_n  ,\wt Z_{k}>\delta^{-1/2} \right)=:\wt E_{2,\delta,n}.
\end{align*}
The goal is to show that
\equh\label{eq:E_2dn}
\lim_{\delta\downarrow0}\limsupn\wt E_{2,\delta,n}=0.
\eque
Introduce 
\[
\varphi_{n,k,\delta}:= \frac1{\wb F_\varepsilon(a_n)}\proba\pp{\varepsilon\wt Z_k/a_n>\kappa, \wt Z_k>\delta^{-1/2}} = \esp \pp{\frac{\wb F_\varepsilon\spp{\kappa a_n/\wt Z_k}}{\wb F_\varepsilon(a_n)}\inddd{\wt Z_k>\delta^{-1/2}}}.
\]
Then $p_{n,k}:=\esp  J_{n,k}/\esp K_n$, $k\in\N$,  yield a probability measure on $\N$. Let $\what Q_n$ be a random variable with such a law, independent from all other random variables. Then,
\[
\wt E_{2,\delta,n} = \esp K_n\cdot \wb F_\varepsilon(a_n) \cdot \esp \varphi_{n,\what Q_n,\delta}.
\]
Recall that
\equh\label{eq:kn}
\limn \frac{\esp K_n}{\nu(n)} = \Gamma(1-\beta) \qmand \limn\esp K_n\cdot\wb F_\varepsilon(a_n) = 1,
\eque
where the second part follows from the first and our assumption on $a_n$ in \eqref{eq:an}.  
We shall argue that
\equh\label{eq:DCT}
\limn \esp \varphi_{n,\what Q_n,\delta} =\kappa^{-\alpha}\esp \pp{\wt Z_{Q_\beta}^\alpha\inddd{\wt Z_{Q_\beta}>\delta^{-1/2}}}.
\eque
This and \eqref{eq:kn} shall then conclude the proof of \eqref{eq:E_2dn}. 
With a little abuse of language, we assume in addition that $\what Q_n\to Q_\beta$ almost surely (strictly speaking, we could always construct them in another probability space, which is enough for the proof). Then
\[
\limn\frac{\wb F_\varepsilon(\kappa a_n/\wt Z_{\what Q_n})}{\wb F_\varepsilon(a_n)}\inddd{\wt Z_{\what Q_n}>\delta^{-1/2}} = \kappa^{-\alpha}\wt Z_{Q_\beta}^\alpha\inddd{\wt Z_{Q_\beta}>\delta^{-1/2}}, \mbox{ almost surely.}
\]
The desired \eqref{eq:DCT} is then the convergence of the corresponding expectation of the above, and it suffices to prove uniform integrability. Namely, we shall show that for some $\epsilon_1>0$ and
\equh\label{eq:UI}
\wb\varphi_n:=  \esp \pp{\pp{\frac{\wb F_\varepsilon\spp{\kappa a_n/\wt Z_{\what Q_n}}}{\wb F_\varepsilon(a_n)}}^{1+\epsilon_1}\inddd{\wt Z_{\what Q_n}>\delta^{-1/2}}} \le C<\infty \mfa n\in\N.
\eque
By Potter's bound  \citep[Proposition 1.5.6]{bingham87regular}, for some $\alpha_+\in(\alpha,\alpha')$ (depending on $\epsilon_1$, which can be arbitrarily small)
\[
\esp \pp{\pp{\frac{\wb F_\varepsilon\spp{\kappa a_n/\wt Z_{k}}}{\wb F_\varepsilon(a_n)}}^{1+\epsilon_1}\inddd{\wt Z_k>\delta^{-1/2}}} \le C \esp \wt Z_k^{\alpha_+}
 \mfa n\in\N,\, k\in\N.
\]
whence
\[
\wb\varphi_n  \le \frac C{\esp K_n} \sif k1 \esp J_{n,k}\esp \wt Z_k^{\alpha_+} = C\esp \wt Z_{\what Q_n}^{\alpha_+}.
\]
We shall compare $\esp J_{n,k}/\esp K_n$ with $p_k\topp\beta$ 
(recall \eqref{eq:Sibuya}), and eventually show that
\equh\label{eq:Q_beta_-}\esp \wt Z_{\what Q_n}^{\alpha_+}\le C\pp{1+\esp \wt Z_{Q_{\beta_-}}^{\alpha_+}}
\eque
for some $\beta_-\in(0,\beta)$. Then, under the assumption \eqref{eq:cond_sd}, we can pick $\beta_-<\beta$ and $\alpha_+>\alpha$ so that the right-hand side above is finite, whence \eqref{eq:UI} and \eqref{eq:E_2dn} follow. To show \eqref{eq:Q_beta_-}, introduce
\[
F_{n,k}(x) := \pp{\frac nx}^ke^{-n/x} \qmand 
f_k(x):=\frac d{dy}(x^ke^{-x}) = (k-x)y^{k-1}e^{-x}.
\]
Note that $F_{n,k}(n/x)$ does not depend on $n$ and that $(d/dx)(F_{n,k}(n/x))=f_k(x)$.
Then,
\begin{align}
\esp J_{n,k} &= \frac1{k!}\int_0^\infty F_{n,k}(x)\nu(dx)
 = \frac1{k!}\int_{1/p_1}^\infty \frac d{dx}F_{n,k}(x)\nu(x)dx \label{eq:EJ_nk}\\
 &= \frac{\nu(n)}{k!}\int_0^{np_1}\frac d{dy}\pp{F_{n,k}\pp{\frac ny}}\frac{\nu(n/y)}{\nu(n)}dy = \frac{\nu(n)}{k!}\int_0^{np_1}f_{k}(y)  y^{-\beta}\frac{L(n/y)}{L(n)}dy.\nonumber
\end{align}
Further, 
\equh\label{eq:pk_beta}
 \Gamma(1-\beta) p_k^{(\beta)}=\frac{\beta \Gamma(k-\beta)}{k!}=\frac{1}{k!}\int_0^\infty y^ke^{-y}
 \beta y^{-\beta-1}dy=\frac{1}{k!}\int_0^\infty
f_{k}(y)y^{-\beta}dy.
\eque
Note that we cannot compare the two directly as $f_{k}$ is not non-negative. Instead we write
\[
\frac1{\esp K_n}\sif k1 \esp J_{n,k}\esp\wt Z_k^{\alpha_+} = \frac {\nu(n)}{\esp K_n} \int_0^{np_1}\sif k1 \frac{(k-y)y^{k-1-\beta}}{k!}\esp \wt Z_k^{\alpha_+}e^{-y}\frac{L(n/y)}{L(n)}dy,
\]
and deal with  the integral  over $[0,1]$ and $[1,np_1]$, respectively. First, using that $\esp \wt Z_k^{\alpha_+}\le k\esp Z^{\alpha_+}$, 
\equh\label{eq:01}
\int_0^1 \sif k1 \frac{ (k-y)y^{k-1-\beta} }{k!}\E\wt Z_{k}^{\alpha_+} e^{-y}\frac{L(n/y)}{L(n)}dy \le C
\int_0^1 \sif k1 \frac{ky^{k-1-\beta_+}}{(k-1)!}e^{-y}dy\le
C\int_0^1 y^{-\beta_+}dy \le C,
\eque
for some $\beta_+\in(\beta,1)$, where in the first step we also applied Potter's bound. Second, for the integral over $[1,np_1]$, we 
shall use the identity, for any increasing sequence of numbers 
$\{D_k\}_{k\in\N_0}, D_0 = 0$, 
\equh\label{eq:Olivier}
\sif k1 \frac{(k-y)y^k}{k!}D_k = \sif k0
\frac{y^{k+1}}{k!}(D_{k+1}-D_{k}).
\eque
Then 
\begin{align*}
  \int_1^{np_1}\sif k1 \frac{(k-y)y^{k-1-\beta}}{k!}\esp \wt Z_k^{\alpha_+}e^{-y}\frac{L(n/y)}{L(n)}dy 
  & = \int_1^{np_1}\sif k1 \frac{y^{k-\beta}}{k!}\pp{\esp\wt Z_{k+1}^{\alpha_+}-\esp \wt Z_k^{\alpha_+}}e^{-y}\frac{L(n/y)}{L(n)}dy \\
  &  \le C\int_1^{np_1}\sif k0 \frac{y^{k-
  \beta_-}}{k!}\pp{\esp \wt Z_{k+1}^{\alpha_+}-\esp \wt Z_k^{\alpha_+}}e^{-y}dy\\
  & = C\int_1^{np_1}\sif k1 \frac{(k-y)y^{k-1-\beta_-}}{k!}  
  \esp \wt Z_k^{\alpha_+}
    e^{-y}
  dy,
\end{align*}
for some $\beta_-<\beta$,
where we applied \eqref{eq:Olivier} in the first and the third steps, and Potter's bound in the second.
The last expression above is then bounded from above by
$C\sif k1 p_k\topp{\beta_-} \esp \wt Z_k^{\alpha_+}  = C \esp \wt Z_{Q_{\beta_-}}^{\alpha_+}$.
Combined with \eqref{eq:01}, we have shown \eqref{eq:Q_beta_-}.

\medskip
\noindent 3) We have
\begin{align*}
 E_{3,\delta,n}
 &\le\P\left(\bigcup_{i= 1}^{N(n)}\left\{\varepsilon_{Y_i}Z_i>\kappa a_n ,\, Z_i<\delta\right\}\right)\le \P\left(\bigcup_{i= 1}^{\infty
 }\left\{\varepsilon_{Y_i}>\kappa \delta^{-1}a_n\right\}\right)\\
&\le \P\left(\bigcup_{\ell= 1}^{K_n}\left\{\varepsilon_{\ell}>\kappa \delta^{-1}a_n\right\}\right)
 \le \esp K_n\cdot \wb F_\varepsilon(\kappa \delta^{-1}a_n).
\end{align*}
It then follows that  
$\limsupn E_{3,\delta,n}\le \kappa^{-\alpha}\delta^{\alpha}$ which converges to $0$ as $\delta\to 0$.

\medskip
\noindent 4) This time,
\begin{align*}
 E_{4,\delta,n}
 &\le \P\left(\bigcup_{i=1}^{N(n)}\left\{\varepsilon_{Y_i}>\delta a_n,\, Z_i>\delta^{-2}\right\}\right)\le \P\left(\bigcup_{\ell=1}^{N(n)}\left\{\varepsilon_\ell>\delta a_n,\, \wt Z_{K_{n,\ell}}>\delta^{-2}\right\}\right)\\
 &\le \sif k1\E J_{n,k}\wb F_\varepsilon(\delta a_n)\wb F_{\wt Z_k}(\delta^{-2}) = \esp K_n\cdot\wb F_\varepsilon(\delta a_n) \wb F_{\wt Z_{\what Q_n}}(\delta^{-2}).
\end{align*}
Then, 
$\limsupn E_{4,\delta,n} \le \delta^{-\alpha}\wb F_{\wt Z_{Q_\beta}}(\delta^{-2})\le \delta^\alpha\E\wt Z_{Q_\beta}^\alpha\to 0$ as $\delta\downarrow 0$.
Thus, \eqref{eq:Lambda_approx} is established and the proposition  is proved.
\end{proof}
\subsection{Critical regime}
\label{sec:critical}
Here we assume $\alpha=\alpha'\beta$. We introduce the following technical assumptions before stating the main theorem in this regime.
Recall our notation for $\wt Z_W$ in \eqref{eq:ZW}. In particular, $\wb F_{\wt Z_{Q_\beta}}(x) = \wb F_Z(x)^\beta$, and  in the subscript of $\wb F_{\varepsilon\wt Z_{Q_\beta}}$ below $\varepsilon$ and $\wt Z_{Q_\beta}$ 
  are understood as independent.
We shall need the following assumption that strengthens $\wb F_\varepsilon(x)\in RV_{-\alpha}$ and $\wb F_Z(x)\in 
RV_{-\alpha'}$.
\begin{assumption}\label{assump:critical}~
 \begin{enumerate}[(i)]
\item $\varepsilon$ has a probability density function $x^{-\alpha-1}l_\varepsilon(x)$, that satisfies
\equh\label{eq:roughly_increasing} 
\limsup_{x\to\infty}\sup_{y\in[x_\varepsilon,x]}\frac{l_\varepsilon(y)}{1\vee l_\varepsilon(x)}<\infty \mbox{ for some }x_\varepsilon\ge 0. 
\eque
\item $\wb F_Z(x) = x^{-\alpha'}L_Z(x)$ with  
\equh\label{eq:roughly_increasingZ1}
\limsup_{x\to\infty}\sup_{y\in[x_Z,x]}\frac{L_Z(y)}{1\vee L_Z(x)}<\infty \mbox{ for some }x_Z\ge 0.
\eque
\item As $x\to\infty$, 
\equh\label{eq:new_condition}
\max\ccbb{\wb F_\varepsilon\pp{x(1\wedge L_Z^{-1/\alpha'}(x))}, x^{-\alpha},\wb F_Z(x)^\beta} = o\pp{\wb F_{\varepsilon\wt Z_{Q_\beta}}(x)}.
\eque
\item $\nu(x) \sim Cx^\beta$ for some constant $C\in(0,\infty)$ (i.e.~$L(x)$ in \eqref{eq:nu} has a  limit in $(0,\infty)$ as $x\to\infty$).
\end{enumerate}
\end{assumption}

\begin{theorem}\label{thm:cr}
Under Assumption \ref{assump:critical} with $\alpha = \alpha'\beta$ and $\{b_n\}_{n\in\N}$ satisfying
\equh\label{eq:bn}
\limn\Gamma(1-\beta)\nu(n)\wb F_{\varepsilon \wt Z_{Q_\beta}}(b_n) =1,
\eque
we have
\[
\sum_{i=1}^{N(n)}\ddelta{\varepsilon_{Y_i}Z_i/b_n,U_i}\weakto \sif \ell1\ddelta{S_\beta^{1/\alpha'}\Gamma_\ell^{-1/\alpha'},U_{\ell}}
\]
as $n\to\infty$ in
$\mathfrak M_p((0,\infty]\times E)$,
and
\equh\label{eq:limit_critical}
\frac1{b_n}M_n\weakto S_\beta^{1/\alpha'}\cdot\calM^{\rm is}_{\alpha'}
\eque
as $n\to \infty$ in $\SM(E)$,
where $S_\beta$ is a totally skewed $\beta$-stable random variable, independent from $\calM^{\rm is}_{\alpha'}$. 
\end{theorem}
\begin{remark}\label{rem:logistic}
In \eqref{eq:limit_critical},
the limit is known as the {\em $(\alpha,\beta)$-logistic  random sup-measure} on 
the metric space $E$
with control measure $\mu$, denoted by 
$\mabl \eqd S_\beta^{\beta/\alpha}\cdot \calM_{\beta/\alpha}^{\rm is}$ 
below for the discussions.
This random sup-measure has appeared in recent literature \citep{molchanov16max,stoev19exchangeable}  (more details of what follows below can be found therein). 
However, we are unaware of 
any example that $\mabl$ arises from the extremes of a stationary sequence. 
It is an $\alpha$-Fr\'echet random sup-measure, with an equivalent series representation as
\[
\mabl(\cdot) \eqd \sup_{\ell\in\N}J_\ell^{\beta/\alpha}\calM_{\beta/\alpha,\ell}^{\rm is}(\cdot),
\]
where 
\[
\calJ 
:=\sif\ell1 \delta_{J_\ell}\sim \PPP\pp{(0,\infty],\Gamma(1-\beta)\inv\beta x^{-\beta-1}dx},
\]
(corresponding to the jumps of a standard $\beta$-stable subordinator up to time 1; in particular $S_\beta\eqd\sif\ell1 J_\ell$)
 and $\{\calM_{\alpha',\ell}^{\rm is}\}_{\ell\in\N}$ 
 (recall $\alpha' = \alpha/\beta$)
  are i.i.d.~copies of $\calM_{\alpha'}^{\rm is}$, independent from $\calJ$.
Moreover, 
\equh\label{eq:fdd}
\proba\pp{\mabl(A_i)\le x_i, i=1,\dots,d} = \exp\pp{-\pp{\summ i1d \frac{\mu(A_i)}{
x_i^{\alpha/\beta}
}}^\beta},\; x_1,\dots,x_d>0,
\eque
for all disjoint $A_i\in\calE$, and the joint law is known as the multivariate logistic extreme-value distribution.  
This family of distributions was  first considered by \citet{gumbel60bivariate} (see e.g.~\citep{fougeres13dense} for more references and some recent developments).
A combinatorial structure underlying the logistic Fr\'echet random sup-measure was recently pointed out in \citep[Remark 3.5]{stoev19exchangeable}, where the name {\em sub-max-stable} was also used (in parallel to sub-stable processes \citep{samorodnitsky94stable}).
We list some properties of the logistic random sup-measures here. 
From \eqref{eq:fdd}, it is immediately seen that $\mabl$ is exchangeable in the sense that $\{\mabl(A_i)\}_{i=1,\dots,d}$ have the same joint law for all disjoint $\{A_i\}_{i=1,\dots,d}$ with the same values $\{\mu(A_i)\}_{i=1,\dots,d}$; when defined on $\Rd$ with $\mu$ being the Lebesgue measure, it is also translation-invariant and self-similar in the usual sense. 
We also mention the following relation that is close to (but not) an invariance property.
For $\gamma,\beta\in(0,1), \alpha>0$, 
\[
S_\gamma^{1/\alpha}\cdot\mabl \eqd \calM_{\alpha\gamma,\beta\gamma}^{\rm lo},
\]
which follows from \eqref{eq:fdd} by conditioning on $S_\gamma$ first (the skewed $\beta$-stable random variable $S_\gamma$ is independent from $\mabl$).  Some simulation examples are provided in Figure \ref{fig:loRSM}. 
\end{remark}

\begin{remark}\label{rem:lL}
The assumption  $\alpha = \alpha'\beta$
 says that $\wb F_\varepsilon(x),\, \wb F_{\wt Z_{Q_\beta}}(x) \in RV_{-\alpha}$, which then implies that  $\wb F_{\varepsilon \wt Z_{Q_\beta}}(x)\in RV_{-\alpha}$; if in addition $\esp \varepsilon^\alpha=\infty$ and $\esp \wt Z_{Q_\beta}
^\alpha
 = \infty$, then 
$\wb F_\varepsilon(x),\, \wb F_{\wt Z_{Q_\beta}}(x) = o(\wb F_{\varepsilon\wt Z_{Q_\beta}}(x))$ (see e.g.~\citep{jessen06regularly}). The latter is slightly weaker than the assumption \eqref{eq:new_condition} which, 
in the presence of \eqref{eq:roughly_increasing} and \eqref{eq:roughly_increasingZ1}, 
is simplified as follows:
\[
\begin{cases}
\wb F_\varepsilon\pp{x(1\wedge L_Z^{-1/\alpha'}(x))} = o\pp{\wb F_{\varepsilon \wt Z_{Q_\beta}}(x)}, & \mbox{ if } \liminf_{x\to\infty}l_\varepsilon(x)>0,\\
\wb F_Z(x)^\beta = o\pp{\wb F_{\varepsilon \wt Z_{Q_\beta}}(x)}, & 
\mbox{ if } \lim_{x\to\infty}l_\varepsilon(x)=0, \liminf_{x\to\infty}L_Z(x)>0,\\
x^{-\alpha}   = o\pp{\wb F_{\varepsilon \wt Z_{Q_\beta}}(x)},& \mbox{ if } \lim_{x\to\infty}l_\varepsilon(x) = 0, \lim_{x\to\infty}L_Z(x) = 0.
\end{cases}
\]
Indeed, it suffices to express
\[
\wb F_\varepsilon\pp{x(1\wedge L_Z^{-1/\alpha'}(x))}\sim \max\ccbb{x^{-\alpha}, \wb F_Z(x)^\beta}\cdot \alpha l_\varepsilon\pp{x(1\wedge L_Z^{-1/\alpha'}(x))}.
\]
\end{remark}

\begin{remark}
Assume $\nu(x)\sim Cx^{\beta}$. 
if both $\varepsilon$ and $Z$ have densities that
 are asymptotically power laws (so that $l_\varepsilon, L_Z$ each has a limit in $(0,\infty)$), then \eqref{eq:roughly_increasing}, \eqref{eq:roughly_increasingZ1} and \eqref{eq:new_condition} hold, and more precisely we have, for some constants $C_1,C_2,C_3>0$,
\[
\wb F_\varepsilon(x)\sim C_1\wb F_Z(x)^\beta = C_1\wb F_{\wt Z_{Q_\beta}}(x) \sim C_2x^{-\alpha}\qmand \wb F_{\varepsilon \wt Z_{Q_\beta}}(x) 
 \sim C_3x^{-\alpha}\log x.
\]

For our proof, the assumption $\nu(x)\sim Cx^\beta$ cannot be relaxed. 
Assumption \eqref{eq:roughly_increasing}  relaxes the asymptotic power-law behavior of the density. 
(A similar 
comment
 applies to \eqref{eq:roughly_increasingZ1} and $L_Z$.)
 In the case $l_\varepsilon(x)\to\infty$, \eqref{eq:new_condition} in addition restricts  $l_\varepsilon$ from increasing too fast.  
As an example, in the special case $L_Z(x) = l_\varepsilon(x)^{1/\beta}$, \eqref{eq:new_condition} becomes
\equh\label{eq:new_condition1}
\wb F_\varepsilon(xl_\varepsilon^{-1/\alpha}(x)) \sim \alpha\inv \wb F_\varepsilon(x)l_\varepsilon(x l_\varepsilon^{-1/\alpha}(x)) = o\pp{\wb F_{\varepsilon\varepsilon'}(x)}, 
\eque
where $\varepsilon'$ is an independent copy of $\varepsilon$ 
and it can be verified that the condition \eqref{eq:new_condition1} remains satisfied with for example $l_\varepsilon(x)\sim C\log ^\gamma x$ for any $\gamma>0$. 
For an example that violates \eqref{eq:new_condition1}, consider $l_\varepsilon(x) = L_\gamma(x) = c_0\exp(\log^\gamma x)$ for $x\ge1$ and 0 otherwise. It is known that for $\theta\in\R$ \citep[P.303]{bojanic71slowly}
\[
\lim_{x\to\infty}\frac{L_\gamma(xL_\gamma^\theta(x)) }{L_\gamma(x)} = \begin{cases}
1 & \mbox{ if } \gamma\in(0,1/2),\\
\exp(\theta\gamma) & \mbox{ if } \gamma = 1/2.
\end{cases}
\]
So, with $\gamma\in(0,1/2]$, $\wb F_\varepsilon(xl_\varepsilon^{-1/\alpha}(x)) \sim C \wb F_\varepsilon (x)l_\varepsilon(x)$. On the other hand, we have
$L_\gamma(y)L_\gamma(x/y)\le CL_\gamma^{2^{1-\gamma}}(x)$ for all $y\in(0,x)$,
$x>1$ (since $a^\gamma+b^\gamma\le 2^{1-\gamma}(a+b)^\gamma$ for $a,b>0,\gamma\in(0,1]$), and hence
\begin{align*}
\wb F_{\varepsilon\varepsilon'}(x)& \le \int_{1}^x y^{-\alpha-1}l_\varepsilon(y)\wb F_\varepsilon(x/y)dy + \wb F_\varepsilon(x)
\le C x^{-\alpha}\int_{1}^xy\inv l_\varepsilon(y)l_\varepsilon(x/y)dy + \wb F_\varepsilon(x) \\
&\le C\wb F_\varepsilon(x)l_\varepsilon(x)\left( l_\varepsilon^{2^{1-\gamma}-2}(x)\log x + l_\varepsilon(x)^{-1}\right)
= o(\wb F_\varepsilon(xl_\varepsilon^{-1/\alpha}(x))).
\end{align*}
\end{remark}

In preparation for the proof of Theorem~\ref{thm:cr}, we introduce
the point 
processes on $(0,\infty]$,
\equh\label{eq:eta}
 \eta_n:=\sum_{i=1}^{N(n)}\delta_{\varepsilon_{Y_i}Z_i/b_n} \qmand \eta:=\sif \ell1\delta_{S_\beta^{1/\alpha'}\Gamma_\ell^{-1/\alpha'}}.
\eque
Again we omit the variables $U$ for the locations. Let $f$ be a continuous non-negative function with compact support in $[\kappa,\infty]$, $\kappa>0$.
Write $\eta_n(f) = \int fd\eta_n$ and similarly for $\eta(f)$. The goal is to show
\[
\limn \E  e^{-\eta_n(f)}= \esp e^{-\eta(f)} 
=\exp\pp{-\Cabf}
\mwith \Cabf = \left(\int_0^\infty (1-e^{- f(v)})\alpha'v^{-\alpha'-1}dv\right)^\beta.
\]
Note that for the proof 
of Theorem~\ref{thm:cr}
we proceed by computing the Laplace functional instead of checking the widely applicable condition due to Kallenberg (\citep[Theorem 4.18]{kallenberg17random}, 
 \citep[Proposition 3.22]{resnick87extreme}), which consists of checking the convergence of 
 $\proba(\eta_n([x,\infty]) = 0)$ and 
 $\esp \eta_n([x,\infty])$ 
 for all $x>0$.  
 The reason that this method does not apply here is that in the limit, we have 
 $\esp\eta([x,\infty]) = \esp (x^{-\alpha'}S_\beta) =\infty$, violating one of the assumptions.

We have,
\begin{align*}
 \E e^{-\eta_n(f)}
&=\E\left(\prod_{i=1}^{N(n)} \exp\left(-f(\varepsilon_{Y_i}Z_i/b_n)\right) \right)=\E\left(\E\left( \prodd k1\infty \prod_{\ell :  K_{n,\ell}=k}  \exp\left(- \sum_{i=1}^k f(\varepsilon_{\ell}Z_{\ell,i}/b_n)\right)\mmid \cY\right)\right)\\
&=\E\left(\prodd k1\infty \prod_{\ell :  K_{n,\ell}=k}  \E\left( \exp\left(- \sum_{i=1}^k f(\varepsilon_{\ell}Z_{\ell,i}/b_n)\right)\right)\right).
\end{align*}
Therefore, recalling that $J_{n,k}=\sif\ell1\inddd{K_{n,\ell} = k}$ and writing that
\[
\psi_{n,k}\equiv \psi_{n,k}(f):=\E \exp\left(- \sum_{i=1}^k f(\varepsilon_{1}Z_{i}/b_n)\right),
\]
we infer that
\[
 \E e^{-\eta_n(f)}
=\E\left(\prod_{k=1}^\infty\left(\E \exp\left(- \sum_{i=1}^k f(\varepsilon_{1}Z_{i}/b_n)\right)\right)^{J_{n,k}}\right)=\E\exp\left(\sif k1 J_{n,k}\log\psi_{n,k}\right).
\]
The proof proceeds by a series of approximations. Consider
\begin{align*} 
\what\Psi_n(f) &:= -\sif k1 J_{n,k}\log\psi_{n,k}, \\
 \wt\Psi_n(f) &:= \sif k1 J_{n,k}(1-\psi_{n,k}), \\
 \bar\Psi_n(f) &:= \sif k1 \E J_{n,k} (1- \psi_{n,k}),\\
 \Psi_n(f) &:= \Gamma(1-\beta)\nu(n)\sif k1 p_k^{(\beta)} (1- \psi_{n,k}).
\end{align*}
Heuristically, the approximation makes sense as {\em for every $k$ fixed}, $\psi_{n,k}\to 1$ and hence $\log \psi_{n,k}\sim \psi_{n,k}-1$, whereas $J_{n,k}$, $\esp J_{n,k}$ and $\Gamma(1-\beta)\nu(n)p_k\topp\beta$ are asymptotically equivalent (recall expressions of the last two in \eqref{eq:EJ_nk} and \eqref{eq:pk_beta}). The uniform control  in $k$ of these equivalences, in an appropriate sense, turned out to be quite involved.

We start with the relatively easy part that $\limn\E e^{-\Psi_n(f)}= \esp e^{-\eta(f)}$, as the following lemma shows. 
Note that here we need slightly weaker assumptions on $\varepsilon$ and $Z$ than Assumption \ref{assump:critical} (see Remark \ref{rem:lL}). 
\begin{lemma}\label{lem:psi_n}
For $\eta$ given as in \eqref{eq:eta},
\[
\esp e^{-\eta(f)} =  e^{-\Cabf}.
\]
If $\nu(x)\sim Cx^\beta$ for some $C\in(0,\infty)$, $\wb F_\varepsilon(x)\in RV_{-\alpha}, \wb F_Z(x)\in RV_{-\alpha'}$, and $\wb F_\varepsilon(x) = o(\wb F_{\varepsilon \wt Z_{Q_\beta}}(x))$, then with $b_n$ as in \eqref{eq:bn},
\[
\limn \Psi_n(f) =   \Cabf.
\]
\end{lemma}
\begin{proof}
Conditionally on $S_\beta$, express points from $\eta$ that are in the intervals $[\kappa,\infty]$ as $\summ i1{N_*}\delta_{U_{i}^{-1/\alpha'}\kappa}$: then $N_*$ is Poisson distributed with parameter $S_\beta\kappa^{-\alpha'}$, and $\{U_{i}\}_{i\in\N}$ are i.i.d.~random variables uniformly distributed over $(0,1)$. So,
\begin{align*}
\esp e^{-\eta(f)} & = \esp \pp{\esp\exp\pp{-
\summ i1{N_*}f(U_i^{-1/\alpha'}\kappa)}\mmid S_\beta} = \esp \exp\pp{\frac{S_\beta}{\kappa^{\alpha'}}\pp{\esp e^{-f(U^{-1/\alpha'}\kappa)}-1}}\\
& = \exp\pp{-\pp{\frac1{\kappa^{\alpha'}}\int_1^\infty
(1-e^{-f(u\kappa)})\alpha' u^{-\alpha'-1}du}^\beta} = e^{-\Cabf}.
\end{align*}
For the second part, we start by writing
\[
 \sif k1 p_k^{(\beta)} (1-\psi_{n,k})
= \sif k1  p_k^{(\beta)}  \left( 1- \E\bb{\E\left(e^{- f(\varepsilon Z/b_n)}\mmid \varepsilon \right)}^k\right)=1- \E\left(\E\left(e^{- f(\varepsilon Z/b_n)}\mmid \varepsilon \right)^{Q_\beta} \right),
\]
where $Q_\beta$ a Sibuya random variable ($\proba(Q_\beta = k) = p_k\topp\beta$), independent of all the rest.
Using $\esp z^{Q_\beta} = 1-(1-z)^\beta$ for $z\in[0,1]$, we get
\[
\sif k1 p_k^{(\beta)} (1-\psi_{n,k})
= \E{\left(1- \E\left(e^{- f(\varepsilon Z/b_n)}\mmid \varepsilon \right)\right)^\beta}=\int_0^\infty \left(1-\esp e^{- f(x Z/b_n)}\right)^\beta dF_\varepsilon(x).
\]
Introduce $Z_{n,x}$ as a random variable with law  determined by
\[
\proba(Z_{n,x}>y) = \proba\pp{Z >\frac{\kappa b_n}x \cdot y\mmid Z>\frac{\kappa b_n}x},
\quad y \ge 1,
\]
and
\[
a_{n,x}(f) :=1-\esp e^{-f(\kappa
Z_{n,x})}.
\]
So we have (recalling that $f$ is supported over $[\kappa,\infty]$)
\equh
 1-\esp e^{-f(xZ/b_n)} = a_{n,x}(f)\cdot \wb F_Z\pp{\kappa b_n/x}.\label{eq:2}
\eque
It follows from $\wb F_Z(x)\in RV_{-\alpha'}$ that, for every $x>0$ fixed, 
\[
\limn a_{n,x}(f)=\int_1^\infty(1-e^{- f(\kappa v)})\alpha'v^{-\alpha'-1}dv = \kappa^{\alpha'}\int_0^\infty (1-e^{- f(v)})\alpha'v^{-\alpha'-1}dv= 
\pp{\kappa^\alpha\Cabf}^{1/\beta},
\]
and for all $\epsilon>0$ we can take $d_\epsilon>0$ small enough so that 
\[
\limsupn\sup_{x\in(0,d_\epsilon b_n]}\abs{a_{n,x}^\beta(f) - \kappa^\alpha\Cabf} \le \epsilon.
\]
We then have
\begin{multline*}
\abs{
\int_0^\infty \wb F_{\wt Z_{Q_\beta}}(\kappa b_n/x) a_{n,x}^\beta(f)  dF_\varepsilon(x)
- \kappa^\alpha \Cabf\wb F_{\varepsilon\wt Z_{Q_\beta}}(\kappa b_n)}\\
\le \int_0^{d_\epsilon b_n}\wb F_{\wt Z_{Q_\beta}}(\kappa b_n/x)\abs{a_{n,x}^\beta(f)-\kappa^\alpha\Cabf}dF_\varepsilon(x) + (1+\kappa ^\alpha\Cabf)\wb F_\varepsilon(d_\epsilon b_n). 
\end{multline*}
The first term on the right-hand side is bounded by, for $n$ large enough, $2\epsilon \wb F_{\varepsilon\wt Z_{Q_\beta}}(\kappa b_n)$, and the second by $C\wb F_\varepsilon(b_n) = o(\wb F_{\varepsilon\wt Z_{Q_\beta}}(b_n))$. Since $\epsilon>0$ can be arbitrarily small, the above implies that
\[
\int_0^\infty \wb F_{\wt Z_{Q_\beta}}(\kappa b_n/x) a_{n,x}^\beta(f)  dF_\varepsilon(x)
\sim \kappa^\alpha \Cabf\wb F_{\varepsilon\wt Z_{Q_\beta}}(\kappa b_n)\sim \Cabf\wb F_{\varepsilon \wt Z_{Q_\beta}}(b_n). 
\]
To sum up,
\begin{align*}
\Psi_n(f) & = \Gamma(1-\beta)\nu(n)\int_0^\infty \left(1-\esp e^{- f(x Z_{1}/b_n)}\right)^\beta dF_\varepsilon(x) \\
& = \Gamma(1-\beta)\nu(n)\int_0^\infty \wb F_{\wt Z_{Q_\beta}}(\kappa b_n/x) a_{n,x}^\beta(f)  dF_\varepsilon(x)
\sim  \Cabf\cdot\Gamma(1-\beta)\nu(n) \wb F_{\varepsilon\wt Z_{Q_\beta}}(b_n).
\end{align*}
The desired result now follows from \eqref{eq:bn}.
\end{proof}
The hard part of the proof lies in approximating $\what \Psi_n$ by $\Psi_n$, where we shall need a very fine control of $1-\psi_{n,k}$. For this purpose, introduce
\equh\label{eq:wt_bn}
\wt b_n := b_n \pp{1\wedge L_Z^{-1/\alpha'}(b_n)} \qmand \wb F_\varepsilon^*(\wt b_n) := \wb F_\varepsilon(\wt b_n)\vee \wt b_n^{-\alpha}.
\eque
The key of the analysis is the following Lemma \ref{lem:psi}.
\begin{lemma}\label{lem:psi}
Under Assumption \ref{assump:critical}, there exists a constant $C>0$ such that for all $n$ large enough, 
\equh\label{eq:1-psi}
1-\psi_{n,k}\le \bb{C \pp{k^\beta\wb F_\varepsilon^*(\wt b_n) + k \wb F_Z(b_n)}}\wedge 1  \; \mfa k\in\N.
\eque
\end{lemma}
\begin{proof}
We have, by \eqref{eq:2},
\[
1- \psi_{n,k}
 =\int_0^\infty1-\pp{ \E e^{-f(xZ/b_n)}}^kdF_\varepsilon(x)
=\int_0^\infty1-\pp{1-a_{n,x}(f)\wb F_Z\pp{\kappa b_n/x}}^kd F_\varepsilon(x).
\]
Pick $r = 1\wedge(\kappa/x_Z)$ (recall \eqref{eq:roughly_increasingZ1} for $x_Z$). 
Then, the integration over $x>r\wt b_n$ is bounded from above by
$\wb F_\varepsilon(r\wt b_n)  \sim Cr^{-\alpha}\wb F_\varepsilon(\wt b_n)$, and this term can be bounded by, for another constant $C$ large enough,  $Ck^\beta\wb F_\varepsilon(\wt b_n)\le Ck^\beta\wb F_\varepsilon^*(\wt b_n)$ for all $k\in\N$. 
Therefore it suffices to show that integration over $x\in[0,r\wt b_n]$ is of the desired order. 
Recall $x_\varepsilon$ from \eqref{eq:roughly_increasing}.
Then, 
\[
\int_0^{x_\varepsilon}1-\pp{1-a_{n,x}(f)\wb F_Z(\kappa b_n/x)}^kdF_\varepsilon(x)\le Ck \wb F_Z(\kappa b_n/x_\varepsilon) \le Ck\wb F_Z(b_n).  
\]
For the interval $[x_\varepsilon,r\wt b_n]$, we observe that by our choice of $x_z$
in \eqref{eq:roughly_increasingZ1} 
for $n$ large enough
\[
 \sup_{x\in[x_\varepsilon,r \wt b_n]} L_Z(\kappa b_n/x)=\sup_{x\in \bb{(x_z\vee \kappa)\frac{b_n}{\wt b_n},\frac\kappa {x_\varepsilon} b_n}} L_Z(x) \le C (1\vee L_Z(b_n)),
\]
and thus 
for all $x\in[x_\varepsilon,r \wt b_n]$, $a_{n,x}(f)\wb F_Z(\kappa b_n/x)\le \what a (x/\wt b_n)^{\alpha'}$ for some constant $\what a>0$.
Introduce
\[
u = (x/\what b_n)^{\alpha'}\qmwith \what b_n := \what a ^{-1/\alpha'}\wt b_n.
\]
We then arrive at  (we also need $n$ large enough so that $\what a(x/\wt b_n)^{\alpha'}<1$),
\begin{align}
\int_{x_\varepsilon}^{r\wt b_n}1-\pp{1-a_{n,x}(f)\wb F_Z(\kappa b_n/x)}^kdF_\varepsilon(x)  & \le
\int_{x_\varepsilon}^{r\wt b_n}1-\pp{1-\what a (x/\wt b_n)^{\alpha'}}^kdF_\varepsilon(x)\nonumber
\\
& \le  C\int_{ ({x_\varepsilon}/\what b_n)^{\alpha'}}^{ (r\wt b_n/\what b_n)^{\alpha'}}\pp{1-(1-u)^k }
\cdot
 dF_\varepsilon\pp{ u^{1/\alpha'}\what b_n}.\label{eq:5}
\end{align}
Write $dF_\varepsilon(x) = f_\varepsilon(x)dx = x^{-\alpha-1}l_\epsilon(x)dx$.
By \eqref{eq:roughly_increasing}, we have, for $u$ in the domain of the integral above,
\[
f_\varepsilon\pp{u^{1/\alpha'}\what b_n} d\pp{u^{1/\alpha'}\what b_n} \le Cu^{-\beta-1} \wt b_n^{-\alpha}\pp{l_\varepsilon(\wt b_n)\vee 1}du.
\]
So \eqref{eq:5} is bounded from above by, uniformly for all $k\in\N$ and $n$ large enough,
\[
C\pp{\wt F_\varepsilon(\wt b_n)\vee \wt b_n^{-\alpha}}\int_0^1\pp{1-(1-u)^k}u^{-\beta-1}du
 =
C\wb F_\varepsilon^*(\wt b_n)\pp{\frac 1\beta+ kB(k,1-\beta)} \le
C\wb F_\varepsilon^*(\wt b_n)
k^\beta,
\]
where $B(x,y) = \Gamma(x)\Gamma(y)/\Gamma(x+y)$ is the beta function. 
  \end{proof}
\begin{proof}[Proof of Theorem~\ref{thm:cr}]
We know that $\E e^{-\eta_n(f)}=\E e^{-\what\Psi_n(f)}$ and that $\limn e^{-\Psi_n(f)}=e^{- \Cabf}$ (Lemma~\ref{lem:psi_n}).
So, to prove Theorem~\ref{thm:cr}, it is sufficient to prove that $\what\Psi_n(f)-\Psi_n(f)\to 0$ in probability for every fixed $f$, and we drop the dependence on $f$ from here on (note that $|e^{-\Psi_n}-e^{-\what\Psi_n}|\le 2 \wedge |1-e^{\Psi_n-\what\Psi_n}|$).
We shall prove successively that $\bar\Psi_n-\Psi_n\to 0$, $\wt\Psi_n-\bar\Psi_n\to 0$ in probability, and $\what\Psi_n-\wt\Psi_n\to 0$ in probability. 

Notice that Assumption \ref{assump:critical} and the choice of $b_n, \wt b_n$ (see \eqref{eq:bn}, \eqref{eq:wt_bn}) imply that (see Remark \ref{rem:lL})
\equh\label{eq:F*}
\limn \nu(n)\wb F^*_\varepsilon(\wt b_n) = 0 \qmand  \limn
n
\wb F_Z(b_n) = 0.
\eque
These and \eqref{eq:1-psi} in Lemma \ref{lem:psi} play a crucial role in the sequel.

\medskip

\noindent (i) We first show that 
\[
\wb\Psi_n-\Psi_n = \sif k1 \pp{\esp J_{n,k}-\nu(n)\Gamma(1-\beta)p_k\topp\beta}(1-\psi_{n,k})\to 0.
\]
Recall $\esp J_{n,k}$ and $p_k\topp\beta$ in \eqref{eq:EJ_nk} and \eqref{eq:pk_beta}: 
\[
\esp J_{n,k} = \frac{\nu(n)}{k!}\int_0^{np_1}f_{k}(y)  y^{-\beta}\frac{L(n/y)}{L(n)}dy \qmand
\Gamma(1-\beta)\nu(n) p_k^{(\beta)}=\frac{\nu(n)}{k!}\int_0^\infty
f_{k}(y)y^{-\beta}dy,
\]
with $f_{k}(y) = (k-y)y^{k-1}e^{-y}$.

For every $\epsilon>0$, let $A_\epsilon\in(0,p_1)$ be such that 
\equh\label{eq:L}
\limsupn\sup_{y\in[0,nA_\epsilon]}\abs{\frac{L(n/y)}{L(n)}-1}<\epsilon.
\eque
We then write,
\begin{align*}
 \bar\Psi_n-\Psi_n
&= \nu(n)\sum_{k=1}^\infty\frac1{k!}\int_0^{nA_\epsilon}  f_{k}(y)  y^{-\beta}\pp{\frac{L(n/y)}{L(n)}-1}dy\cdot (1-\psi_{n,k})\\
&\quad+ \nu(n)\sum_{k=1}^\infty\frac1{k!}\int_{nA_\epsilon}^{np_1}  f_{k}(y)  y^{-\beta}\pp{\frac{L(n/y)}{L(n)}-1}dy\cdot (1-\psi_{n,k})\\
&\quad-\nu(n)\sum_{k=1}^\infty\frac1{k!} \int_{np_1}^\infty f_{k}(y)  y^{-\beta}dy\cdot (1-\psi_{n,k})=: I^\epsilon_{n,1} + I^\epsilon_{n,2} - I_{n,3}.
\end{align*}
We shall show that 
\equh\label{eq:In2}
\limsupn \frac{|I_{n,1}^\epsilon|}{\Psi_n}\le \frac
{2\epsilon}
 {\Gamma(1-\beta)}\qmand \limsupn \pp{|I_{n,2}^\epsilon|+|I_{n,3}|} = 0, \quad\mfa \epsilon>0.
\eque

We first deal with $I_{n,3}$. Introduce
\[
I_{n,3}(p) := \nu(n)\sum_{k=1}^\infty\frac1{k!} \int_{np}^\infty f_{k}(y)  y^{-\beta}dy\cdot (1-\psi_{n,k}), p>0.
\]
So in \eqref{eq:In2} $I_{n,3} = I_{n,3}(p_1)$. 
Observe that, by integration by part,
\begin{align*}
\frac1{k!}\int_{np}^\infty f_{k}(y)y^{-\beta}dy &= \frac1{k!}y^ke^{-y}y^{-\beta}
\bigg\vert_{np}^\infty + \frac1{k!}\beta\int_{np}^\infty y^ke^{-y}y^{-\beta-1}dy\\
& = \frac1{k!}\beta\int_{np}^\infty y^{k-1-\beta}e^{-y}dy - \frac{(np)^{k-\beta}}{k!}e^{-np}\\
& =\Gamma(1-\beta) p_k^{(\beta)} \P(\gamma_{k-\beta}>np)-\frac{(np)^{k-\beta}}{k!}e^{-np},  
\end{align*}
where $\gamma_{k-\beta}$ is a random variable of Gamma distribution with parameter $k-\beta$. 
Thus,
\equh\label{eq:I_n3}
|I_{n,3}(p)|
\le C\nu(n)\sif k1p_k^{(\beta)} \P(\gamma_{k-\beta}>np)(1-\psi_{n,k}) + e^{-np}\nu(n)\sif k1 \frac{(np)^{k-\beta}}{k!}(1-\psi_{n,k}).
\eque
We deal with the two series separately. 
First, recalling \eqref{eq:F*}, 
one can find a sequence of integers  $\ell_n$ such that 
\[
\ell_n\to\infty, 
\quad \nu(n)\wb F_\varepsilon^*(\wt b_n)\ell_n\to 0
\qmand n \wb F_Z(b_n)\ell_n^{2-\beta}\to 0.
\]
Then, applying Markov inequality $\P(\gamma_{k-\beta}>np)\le(\esp\gamma_{k-\beta}/np)\wedge 1 = ((k-\beta)/np)\wedge 1$ to $k\le n\ell_n$ and $k>n\ell_n$ respectively, 
we have
\begin{align*} 
&\nu(n)\sif k1p_k^{(\beta)} \P(\gamma_{k-\beta}>np)(1-\psi_{n,k}) \\
& \le C\frac{\nu(n)}n\summ k1{n\ell_n}p_k\topp\beta(k-\beta)\pp{k^\beta \wb F_\varepsilon^*(\wt b_n) + k \wb F_Z(b_n)} + C\nu(n)\sif k{n\ell_n+1}p_k\topp\beta\\
& \le C\nu(n)\wb F_\varepsilon^*(\wt b_n)\ell_n + C\frac{\nu(n)}n\wb F_Z(b_n)(n\ell_n)^{2-\beta}+ C\nu(n)(n\ell_n)^{-\beta},
\end{align*}
where in the last step above we use the fact that $p_k\topp \beta\sim Ck^{-\beta-1}$ 
(recall Sibuya distribution \eqref{eq:Sibuya}) and Karamata theorem.
By our assumption on $\ell_n$ we have shown that the first series in \eqref{eq:I_n3} goes to zero. 
The second series in \eqref{eq:I_n3} can be bounded by, using \eqref{eq:1-psi},
\begin{multline}
\label{eq:second}
C\nu(n)e^{-np}\sif k1\frac{(np)^{k-\beta}}{k!}\pp{k^\beta \wb F_\varepsilon^*(\wt b_n) + k \wb F_Z(b_n)}\\
\le C\nu(n)\wb F_\varepsilon^*(\wt b_n)e^{-np}\sif k0\frac{(np)^{k+1-\beta
}}{\Gamma(k+2-\beta)}
+ C\nu(n)\wb F_Z(b_n)(np)^{1-\beta} \le C\nu(n)\wb F_\varepsilon^*(\wt b_n) + Cn\wb F_Z(b_n),
\end{multline}
where in the second inequality, the first term is bounded by the following estimate on  Mittag--Leffler function (e.g.~\citep[Eq.(6)]{gorenflo02computation})
\[
E_{1,2-\beta}(y) = \sif k0 \frac{y^k}{\Gamma(k+2-\beta)} \le C y^{\beta-1}
e^y
 \mfa y\ge 1,
\]
and the second term by the fact $\nu(n)\le Cn^\beta$.
By \eqref{eq:F*}, \eqref{eq:second} tends to zero, and hence $I_{n,3}(p)\to 0$ for all $p>0$.

Now we deal with $I_{n,1}^\epsilon$. 
Note that $f_{k}(y)$ changes sign at $y=k$ so we proceed with caution. First we write, by \eqref{eq:Olivier},
\[
I_{n,1}^\epsilon = \nu(n)\int_0^{nA_\epsilon}\sif k0 \frac{y^{k-\beta}e^{-y}}{k!}(\psi_{n,k}-\psi_{n,k+1})\pp{\frac{L(n/y)}{L(n)}-1}dy,
\]
and recall that $\psi_{n,k}-\psi_{n,k+1}> 0$. 
Then, for $n$ large enough, thanks to \eqref{eq:L},
\begin{align*}
|I_{n,1}^\epsilon| & \le \nu(n)\int_0^{nA_\epsilon}\sif k0\frac{y^{k-\beta}e^{-y}}{k!}(\psi_{n,k}-\psi_{n,k+1})\abs{\frac{L(n/y)}{L(n)}-1}dy\\
& \le 2\epsilon \cdot \nu(n)\int_0^{nA_\epsilon}\sif k0 \frac{y^{k-\beta}e^{-y}}{k!}(\psi_{n,k}-\psi_{n,k+1})dy \\
& = 2\epsilon \cdot
\nu(n)\sif k1\int_0^{nA_\epsilon}\frac{f_{k}(y)y^{-\beta}}{k!}dy \cdot(1-\psi_{n,k})
 =  2\epsilon\frac{\Psi_n}{\Gamma(1-\beta)}  
 - 2\epsilon \cdot I_{n,3}(A_\epsilon),
\end{align*}
where in the third step \eqref{eq:Olivier} is applied again. 
We have seen that  $|I_{n,3}(A_\epsilon)|\to 0$. This proves the first part of \eqref{eq:In2}. 

It remains to deal with $I_{n,2}^\epsilon$. By the same trick on $I_{n,1}^\epsilon$ above using \eqref{eq:Olivier} twice, but this time combined with
$\limsupn\sup_{y\in[nA_\epsilon,np_1]}|L(n/y)/L(n)-1|\le C$ (which cannot be arbitrarily small, but is finite under our assumption on $\nu$), we have that
\begin{align*}
|I_{n,2}^\epsilon| & \le  
 \nu(n)\int_{nA_\epsilon}^{np_1}\sif k0 \frac{y^{k-\beta}e^{-y}}{k!}(\psi_{n,k}-\psi_{n,k+1})\abs{\frac{L(n/y)}{L(n)}-1}dy\\
& \le  
C \nu(n)\int_{nA_\epsilon}^{np_1}\sif k0 \frac{y^{k-\beta}e^{-y}}{k!}(\psi_{n,k}-\psi_{n,k+1})dy \\
&= C\pp{I_{n,3}(A_\epsilon)-I_{n,3}(p_1)}\le C(|I_{n,3}(A_\epsilon)|+|I_{n,3}(p_1)|)\to 0.
\end{align*}
This  completes the proof of $\bar\Psi_n-\Psi_n\to 0$.

\medskip

\noindent (ii) Next, we prove 
\[
\wt \Psi_{n} - \bar\Psi_n = \sif k1 (J_{n,k}-\esp J_{n,k})(1-\psi_{n,k}) \stackrel\proba\to 0.
\]
Introduce $N_\ell(n) :=\summ i1{N(n)}\inddd{Y_i = \ell}$. So $\{N_\ell(n)\}_{\ell\in\N}$ are independent Poisson random variables. Recall that
$J_{n,k} = \sif \ell1 \inddd{N_\ell(n) = k}$. By independence,
\[
\esp \pp{J_{n,k}J_{n,k'}} = \sum_{\ell\ne\ell'}\proba(N_\ell(n) = k)\proba(N_{\ell'}(n) = k') \le \esp J_{n,k}\esp J_{n,k'} \mfa k\ne k'. 
\]
It follows that
$\Var\spp{\wt \Psi_n-\bar \Psi_n}\le \sif k1 \Var J_{n,k}\cdot (1-\psi_{n,k})^2 \le \sif k1 \esp J_{n,k}\cdot (1- \psi_{n,k})^2$. 
Noticing that $J_{n,k}=0$ when $k>N(n)$ and then using \eqref{eq:1-psi} and Cauchy--Schwarz inequality, we infer
\equh\label{eq:var1}
\Var\pp{\wt \Psi_n-\bar \Psi_n}
\le \E\pp{\wt\Psi_n \max_{k=1,\ldots,N(n)}(1- \psi_{n,k})}\le 
\|\wt\Psi_n\|_2\|
1-\psi_{n,N(n)}\|_2.
\eque
Observe that 
$\snn{\wt\Psi_n}_2^2 = \Var \wt\Psi_n + (\esp \wt \Psi_n)^2 \le \wb\Psi_n + \wb\Psi_n^2$. 
Since $\wb \Psi_n$ has a finite limit, 
\eqref{eq:var1} is bounded from above by
\equh\label{eq:control}
C\nn{1-\psi_{n,N(n)}}_2 \le C\nn{N(n)^\beta\wb F_\varepsilon^*(\wt b_n) + N(n)\wb F_Z(n)}_2\to 0,
\eque
as a consequence of \eqref{eq:F*} and the fact that
 $N(n)^\beta$ is of order $\nu(n)\sim C n^\beta$ as $N(n)/n \stackrel{L^2}\to 1$.
Therefore we have proved that $\wt\Psi_{n} - \bar\Psi_n\to 0$ in $L^2$.

\medskip

\noindent(iii) It remains to prove that
$\what \Psi_n - \wt \Psi_n \stackrel\proba\to 0$. 
Using that $J_{n,k}=0$ when $k>N(n)$ and that $|\log(x)+(1-x)|/(1-x)\le(1-x)/x$ for $x\in(0,1)$, for $n$ large enough, we get 
\begin{align*}
 \left|\what\Psi_n-\wt\Psi_n\right|& = \abs{\sif k1 J_{n,k}\pp{\log \psi_{n,k} + 1-\psi_{n,k}}} \le \sif k1 J_{n,k}\cdot (1-\psi_{n,k})\frac{1-\psi_{n,k}}{\psi_{n,k}}\\
 & \le |\wt\Psi_n|\max_{k=1,\ldots,N(n)}\frac{1-\psi_{n,k}}{\psi_{n,k}} = |\wt\Psi_n|\frac{1-\psi_{n,N(n)}}{\psi_{n,N(n)}}.
\end{align*}
Since we have seen that $\wt\Psi_n \stackrel\proba\to \limn\Psi_n \in(0,\infty)$, by \eqref{eq:control},
we infer that $\what\Psi_n-\wt\Psi_n\to 0$ in probability.
\end{proof}
\NN
\section{Extremal limit theorems for the perturbed one-dimensional Karlin model}\label{sec:KSV}
We now apply Section \ref{sec:phase} to the perturbed one-dimensional Karlin model discussed in introduction. Let $\{\varepsilon_i,Y_i,Z_i\}_{i\in\N}$ be as in the Poisson--Karlin model. Then, the perturbed Karlin model is the stationary sequence defined as
\[
X_i:=\varepsilon_{Y_i}Z_i.
\]
We are interested in the empirical random sup-measure defined as
\[\what M_n(\cdot) := \max_{i=1,\dots,n: i/n\in\cdot}\varepsilon_{Y_i}Z_i \mbox{ in $\SM([0,1])$.}
\]
Below, $\calM_{\alpha}^{\rm is}$ denotes an independently scattered $\alpha$-Fr\'echet random sup-measure on $[0,1]$ with Lebesgue control measure (recall \eqref{eq:isRSM}).
\begin{theorem}\label{thm:KSV}
Assume $\alpha,\alpha'>0$ and $\beta\in(0,1)$. 
\begin{enumerate}[(i)]
\item {\em (Signal-dominance regime)} If $\wb F_\varepsilon(x)\in RV_{-\alpha}$, $\nu$ satisfies \eqref{eq:nu}, and $\esp \wt Z_{Q_{\beta-\epsilon}}^{\alpha+\epsilon}<\infty$ for some $\epsilon>0$, then for $\{a_n\}_{n\in\N}$ such that $\limn \Gamma(1-\beta)\nu(n)\wb F_\varepsilon(a_n) = 1$,
\[
  \summ i1n \ddelta{\varepsilon_{Y_i}Z_i/a_n, i/n} \weakto  \sif \ell1\summ i1{Q_{\beta,\ell}}\ddelta{\Gamma_\ell^{-1/\alpha}Z_{\ell,i},U_{\ell,i}}
\]
and
\[
\frac1{a_n}\what M_n(\cdot)\weakto  \mabZ(\cdot):=\sup_{\ell\in\N}\frac1{\Gamma_\ell^{1/\alpha}}\max_{i=1,\dots,Q_{\beta,\ell}}Z_{\ell,i}\inddd{U_{\ell,i}\in\cdot}.
\]
\item {\em (Critical regime)}
If $\alpha = \alpha'\beta$, $\wb F_\varepsilon(x)\in RV_{-\alpha}$, $\wb F_Z(x)\in RV_{-\alpha'}$ and Assumption \ref{assump:critical} holds, then for $\{b_n\}_{n\in\N}$ such that $\limn \Gamma(1-\beta)\nu(n)\wb F_{\varepsilon\wt Z_{Q_\beta}}(b_n) = 1$,
\[
  \summ i1n \ddelta{\varepsilon_{Y_i}Z_i/b_n, i/n} \weakto \sif \ell1\ddelta{S_\beta^{1/\alpha'}\Gamma_\ell^{-1/\alpha'},U_{\ell}}
\]
and
\[
\frac1{b_n}\what M_n\weakto S_\beta^{1/\alpha'}\cdot\calM_{\alpha'}^{\rm is}.
\]
\item {\em (Noise-dominance regime)}
If $\wb F_Z(x)\in RV_{-\alpha'}$ and $\esp_{\vv\varepsilon} \varepsilon_Y^{\alpha'+\epsilon}<\infty$ almost surely for some $\epsilon>0$, then for $\{c_n\}_{n\in\N}$ such that $\limn n\wb F_\varepsilon(c_n) = 1$, 
conditionally on $\vv\varepsilon$,
\[
  \summ i1n \ddelta{\varepsilon_{Y_i}Z_i/c_n, i/n} \weakto 
\sif\ell1\ddelta{\varepsilon_{Y_\ell}\Gamma_\ell^{-1/\alpha'}, U_{\ell}}
\]
and
\[
\frac1{c_n}\what M_n\weakto \pp{\esp_{\vv\varepsilon}\varepsilon_Y^{\alpha'}}^{1/\alpha'}\cdot\calM_{\alpha'}^{\rm is}.
\]
\end{enumerate}
In all three cases, the first convergence in distribution is in $\mathfrak M_p((0,\infty]\times [0,1])$ and the second in $\SM([0,1])$. 
\end{theorem}
\begin{remark}
Our Poissonization method is different from the one applied for the original Karlin model \citep{karlin67central,gnedin07notes}, which is essentially a time-change lemma \citep{billingsley99convergence} that depends crucially on the fact that $\R$ and $\N$ are ordered. Our method is geometry free in the sense that it can be adapted to other situations where the time-change lemma does not apply. For example, one may consider the $\Nd$-extension of the problem: let $\{\varepsilon_\ell\}_{\ell\in\N}$ be as before, $Y$ and $Z$ be i.i.d.~indexed by $\vv i\in\Nd$, all assumed to be independent, and 
\[
\what M_n(\cdot):=\max_{\vv i \in\{1,\dots,n\}^d:\vv i/n\in \cdot}\varepsilon_{Y_{\vv i}}Z_{\vv i}.
\]
Theorem \ref{thm:KSV} can be extended to this model with some obvious changes, and our Poissonization method applies to this model too with little extra effort. We omit the details.
\end{remark}
\subsection{A Poissonization method}
Our method is unified for all three different regimes.
Consider the point process of the perturbed Karlin  model
\[
\what \xi_n := \summ i1n \ddelta{\varepsilon_{Y_i}Z_i/r_n, i/n},
\]
where $r_n = a_n,b_n$ or $c_n$ depending on the regime, 
and we do not write the rate $r_n$ explicitly. 
A natural Poissonization of $\what \xi_n$ would be
\[
\xi_n:=\summ i1{N(n)}\ddelta{\varepsilon_{Y_i}Z_i/r_n,U_i},
\]
which is the same point-process investigated before, with i.i.d.~uniform random variables $\{U_i\}_{i\in\N}$ 
on $[0,1]$. We have seen in Section \ref{sec:phase} that
\[
\xi_n\weakto \xi
\]
in $\mathfrak M_p((0,\infty]\times(0,1))$, where $\xi$ is the Poisson point process underlying the random sup-measure in the corresponding regime. This was actually achieved by  computing,  for $f$  a continuous function on $(0,\infty]\times[0,1]$ with compact support,
\equh\label{eq:xi_n(f)}
\limn\esp e^{-\xi_n(f)} = \esp e^{-\xi(f)} = e^{-\Cabfs},
\eque
for some expression $\Cabfs$ that depends on the regime of interest. We omit the expression. 

Consider $f$ in the form
\equh\label{eq:f_step}
f(x,u) = \summ j1d\theta_j \inddd{x\in(g_j,h_j), u\in(s_j,t_j)},\quad  c_j>0, 0<g_j<h_j, 0\le s_j<t_j\le 1.
\eque
Let $\delta>0$ denote a tuning parameter. Our Poissonization method is summarized by the following lemma.
\begin{lemma}\label{lem:1}
For $f$ as above and $\delta>0$, there exist point processes $\xi_{n,\delta,-},\xi_{n,\delta,+}$ in $\mathfrak M_p((0,\infty]\times[0,1])$ such that
\equh\label{eq:coupling_xi}
\limn\proba\pp{\xi_{n,\delta,-}(f) \le \what \xi_n(f) \le \xi_{n,\delta,+}(f)} = 1,
\eque
and moreover, there exist constants $\mathfrak C_{\alpha,\beta,*}^{\delta,\pm}(f)$ such that
\equh\label{eq:coupling_C}
\limn \esp e^{-\xi_{n,\delta,\pm}(f)} = e^{-\mathfrak C_{\alpha,\beta,*}^{\delta,\pm}(f)} \qmand \lim_{\delta\downarrow 0}\mathfrak C_{\alpha,\beta,*}^{\delta,\pm}(f) = \Cabfs.
\eque
\end{lemma}
\begin{proof}[Proof of Theorem \ref{thm:KSV}]
By the first part of the lemma above and the limit theorem for the Poissonized model, we have
\[
e^{-\mathfrak C^{\delta,+}_{\alpha,\beta,*}(f)}
= \limn\esp e^{-\xi_{n,\delta,+}(f)} 
\le \liminf_{n\to\infty} \esp e^{-\what \xi_n(f)} \le \limsup_{n\to\infty} \esp e^{-\what\xi_n(f)}\le \limn\esp e^{-\xi_{n,\delta,-}(f)} =
e^{-\mathfrak C^{\delta,-}_{\alpha,\beta,*}(f)}.
\]
The second part of Lemma \ref{lem:1} then entails, letting $\delta$ decrease to zero, that  inequality in the middle above is actually an equality, and hence the desired convergence of Laplace functional for $f$ as a step function in \eqref{eq:f_step}. The convergence for general continuous $f$ in \eqref{eq:xi_n(f)} follows by a standard approximation argument.
\end{proof}
\begin{proof}[Proof of Lemma \ref{lem:1}]
To start with, assume in addition that all $A_j:=(s_j,t_j), j=1,\dots,d$ are disjoint. Introduce
\[
n_j := \summ i1n \inddd{i/n\in A_j}.
\]
Then,
\equh\label{eq:coupling1}
\what\xi_n(f) = \summ j1d\theta_j\summ i1n \inddd{i/n\in A_j}\inddd{\varepsilon_{Y_i}Z_i/b_n\in(g_j,h_j)} \eqd \summ j1d\theta_j\sum_{i=1}^{n_j}\inddd{\varepsilon_{Y_{j,i}}Z_{j,i}/b_n\in(g_j,h_j)},
\eque
where $\varepsilon$ is as before, $Y_{j,i}$ and $Z_{j,i}$ are i.i.d.~copies of $Y$ and $Z$, respectively, and independent from $\{\varepsilon_\ell\}_{\ell\in\N}$
 (but $\{\varepsilon_{Y_{j,i}}\}_{j,i}$ are dependent).  

Now we introduce, for every $\delta\in(0,1)$,
\[
\xi_{n,\delta,\pm}\eqd \summ i1{N((1\pm\delta)n)} \ddelta{\varepsilon_{Y_i}Z_i/r_n,U_i},
\]
where $\{\varepsilon_i,Y_i,Z_i,U_i\}_{i\in\N}$ are as in the Poisson--Karlin model 
(with $\mu$ the uniform law on $[0,1]$),
independent from the Poisson random variable $N((1\pm\delta)n)$ (with mean $(1\pm\delta)n$).  
 The above is interpreted as the law of $\xi_{n,\delta,+}$ and $\xi_{n,\delta,-}$ separately. We shall first derive for each of $\xi_{n,\delta,\pm}(f)$  a similar representation as \eqref{eq:coupling1} in \eqref{eq:coupling2} below, and then explain the coupling. Set
\[
N_{n,\delta,\pm}
(j):= \xi_{n,\delta,\pm}((0,\infty)\times A_j), j=1,\dots,d.
\]
Since $\{A_j\}_{j=1,\dots,d}$ are disjoint, $\{N_{n,\delta,+}(j)\}_{j=1,\dots,d}$ ($\{N_{n,\delta,-}(j)\}_{j=1,\dots,d}$ resp.) are independent Poisson random variables 
with parameters $(1+\delta)n |A_j|$ ($(1-\delta)n|A_j|$ resp.). 
We hence arrive at
\equh\label{eq:coupling2}
\xi_{n,\delta,\pm}(f) \eqd \summ j1d\theta_j\sum_{i=1}^{N_{n,\delta,\pm}
(j)}\inddd{\varepsilon_{Y_{j,i}}Z_{j,i}/r_n\in(g_j,h_j)}.
\eque

Now we explain the coupling of $\what \xi_{n}(f),\xi_{n,\delta,+}(f)$ and $\xi_{n,\delta,-}(f)$. In view of \eqref{eq:coupling1} and \eqref{eq:coupling2}, we assume  naturally that the three random summations share the same $\{\varepsilon_i\}_{i\in\N}, \{Y_{j,i},Z_{j,i}\}_{j=1,\dots,d,i\in\N}$, and that these random variables are independent from 
$\{N_{n,\delta,\pm}(j)\}_{j=1,\dots,d}$. It is also natural to assume that $\xi_{n,\delta,+}$ and $\xi_{n,\delta,-}$ are coupled in the sense that the latter is obtained from the former by a standard thinning procedure (with probability $(1-\delta)/(1+\delta)$ to keep independently each point from the former), which leads to $N_{n,\delta,-}(j)\le N_{n,\delta,+}(j)$ almost surely for all $j$. 
Therefore it remains to show
\equh\label{eq:Omega_n}
\limn\proba\pp{N_{n,\delta,-}(j)\le n_j\le N_{n,\delta,+}(j) \mfa j=1,\dots,d} = 1.
\eque
But, since $n_j = \#(nA_j\cap \Z)\sim n|A_j|$, the above follows immediately from the concentration of Poisson random variables $N_{n,\delta,\pm}(j)$ around $(1\pm\delta)n|A_j|$, and the probability approaches one exponentially fast as $n\to\infty$. Combining \eqref{eq:coupling1}, \eqref{eq:coupling2} and \eqref{eq:Omega_n} yields \eqref{eq:coupling_xi}. The part \eqref{eq:coupling_C} follows from our result in the previous section 
and we omit the details.
\end{proof}
\subsection*{Acknowledgements}
The authors would like to thank Anja 
Jan\ss en
 for helpful discussions on stochastic volatility models, and Rafa\l~Kulik for very careful reading of an earlier version of the paper and many helpful comments. 
  The authors would also like to thank two anonymous referees and the Associate Editor for their very helpful comments and suggestions. 
YW's research was partially supported by Army Research Office grant W911NF-17-1-0006.

\bibliographystyle{apalike}

\bibliography{references,references18}

\def\cprime{$'$} \def\polhk#1{\setbox0=\hbox{#1}{\ooalign{\hidewidth
  \lower1.5ex\hbox{`}\hidewidth\crcr\unhbox0}}}
  \def\polhk#1{\setbox0=\hbox{#1}{\ooalign{\hidewidth
  \lower1.5ex\hbox{`}\hidewidth\crcr\unhbox0}}}
\begin{thebibliography}{}

\bibitem[Aldous, 1989]{aldous89probability}
Aldous, D. (1989).
\newblock {\em Probability approximations via the {P}oisson clumping
  heuristic}, volume~77 of {\em Applied Mathematical Sciences}.
\newblock Springer-Verlag, New York.

\bibitem[Andersen et~al., 2009]{andersen09handbook}
Andersen, T.~G., Davis, R.~A., Krei{\ss}, J.-P., and Mikosch, T.~V. (2009).
\newblock {\em Handbook of financial time series}.
\newblock Springer Science \& Business Media.

\bibitem[Basrak et~al., 2002]{basrak02regular}
Basrak, B., Davis, R.~A., and Mikosch, T. (2002).
\newblock Regular variation of {GARCH} processes.
\newblock {\em Stochastic Process. Appl.}, 99(1):95--115.

\bibitem[Basrak et~al., 2018]{basrak18invariance}
Basrak, B., Planini\'{c}, H., and Soulier, P. (2018).
\newblock An invariance principle for sums and record times of regularly
  varying stationary sequences.
\newblock {\em Probab. Theory Related Fields}, 172(3-4):869--914.

\bibitem[Basrak and Segers, 2009]{basrak09regularly}
Basrak, B. and Segers, J. (2009).
\newblock Regularly varying multivariate time series.
\newblock {\em Stochastic Process. Appl.}, 119(4):1055--1080.

\bibitem[Billingsley, 1999]{billingsley99convergence}
Billingsley, P. (1999).
\newblock {\em Convergence of probability measures}.
\newblock Wiley Series in Probability and Statistics: Probability and
  Statistics. John Wiley \& Sons Inc., New York, second edition.
\newblock A Wiley-Interscience Publication.

\bibitem[Bingham et~al., 1987]{bingham87regular}
Bingham, N.~H., Goldie, C.~M., and Teugels, J.~L. (1987).
\newblock {\em Regular variation}, volume~27 of {\em Encyclopedia of
  Mathematics and its Applications}.
\newblock Cambridge University Press, Cambridge.

\bibitem[Bojani\'{c} and Seneta, 1971]{bojanic71slowly}
Bojani\'{c}, R. and Seneta, E. (1971).
\newblock Slowly varying functions and asymptotic relations.
\newblock {\em J. Math. Anal. Appl.}, 34:302--315.

\bibitem[Breiman, 1965]{breiman65some}
Breiman, L. (1965).
\newblock On some limit theorems similar to the arc-sin law.
\newblock {\em Theory of Probability and its Applications}, 10(2):323--331.

\bibitem[Chen and Samorodnitsky, 2020]{chen20extreme}
Chen, Z. and Samorodnitsky, G. (2020).
\newblock Extreme value theory for long-range-dependent stable random fields.
\newblock {\em J. Theoret. Probab.}, 33(4):1894--1918.

\bibitem[Dombry et~al., 2018]{dombry18tail}
Dombry, C., Hashorva, E., and Soulier, P. (2018).
\newblock Tail measure and spectral tail process of regularly varying time
  series.
\newblock {\em Ann. Appl. Probab.}, 28(6):3884--3921.

\bibitem[Drees et~al., 2015]{drees15statistics}
Drees, H., Segers, J., and Warcho\l, M. (2015).
\newblock Statistics for tail processes of {M}arkov chains.
\newblock {\em Extremes}, 18(3):369--402.

\bibitem[Durieu et~al., 2020]{durieu20infinite}
Durieu, O., Samorodnitsky, G., and Wang, Y. (2020).
\newblock From infinite urn schemes to self-similar stable processes.
\newblock {\em Stochastic Process. Appl.}, 130(4):2471--2487.

\bibitem[Durieu and Wang, 2016]{durieu16infinite}
Durieu, O. and Wang, Y. (2016).
\newblock From infinite urn schemes to decompositions of self-similar
  {G}aussian processes.
\newblock {\em Electron. J. Probab.}, 21:Paper No. 43, 23.

\bibitem[Durieu and Wang, 2018]{durieu18family}
Durieu, O. and Wang, Y. (2018).
\newblock A family of random sup-measures with long-range dependence.
\newblock {\em Electronic Journal of Probability}, 23(107):1--24.

\bibitem[Foug\`eres et~al., 2013]{fougeres13dense}
Foug\`eres, A.-L., Mercadier, C., and Nolan, J.~P. (2013).
\newblock Dense classes of multivariate extreme value distributions.
\newblock {\em J. Multivariate Anal.}, 116:109--129.

\bibitem[Fu and Wang, 2020]{fu20stable}
Fu, Z. and Wang, Y. (2020).
\newblock Stable {P}rocesses with {S}tationary {I}ncrements {P}arameterized by
  {M}etric {S}paces.
\newblock {\em J. Theoret. Probab.}, 33(3):1737--1754.

\bibitem[Gnedin et~al., 2007]{gnedin07notes}
Gnedin, A., Hansen, B., and Pitman, J. (2007).
\newblock Notes on the occupancy problem with infinitely many boxes: general
  asymptotics and power laws.
\newblock {\em Probab. Surv.}, 4:146--171.

\bibitem[Gorenflo et~al., 2002]{gorenflo02computation}
Gorenflo, R., Loutchko, J., and Luchko, Y. (2002).
\newblock Computation of the {M}ittag-{L}effler function
  {$E_{\alpha,\beta}(z)$} and its derivative.
\newblock {\em Fract. Calc. Appl. Anal.}, 5(4):491--518.
\newblock Dedicated to the 60th anniversary of Prof. Francesco Mainardi.

\bibitem[Gumbel, 1960]{gumbel60bivariate}
Gumbel, E.~J. (1960).
\newblock Bivariate exponential distributions.
\newblock {\em J. Amer. Statist. Assoc.}, 55:698--707.

\bibitem[Janssen, 2019]{janssen19spectral}
Janssen, A. (2019).
\newblock Spectral tail processes and max-stable approximations of multivariate
  regularly varying time series.
\newblock {\em Stochastic Process. Appl.}, 129(6):1993--2009.

\bibitem[Janssen and Drees, 2016]{janssen16stochastic}
Janssen, A. and Drees, H. (2016).
\newblock A stochastic volatility model with flexible extremal dependence
  structure.
\newblock {\em Bernoulli}, 22(3):1448--1490.

\bibitem[Jessen and Mikosch, 2006]{jessen06regularly}
Jessen, A.~H. and Mikosch, T. (2006).
\newblock Regularly varying functions.
\newblock {\em Publ. Inst. Math. (Beograd) (N.S.)}, 80(94):171--192.

\bibitem[Kallenberg, 2017]{kallenberg17random}
Kallenberg, O. (2017).
\newblock {\em Random measures, theory and applications}, volume~77 of {\em
  Probability Theory and Stochastic Modelling}.
\newblock Springer, Cham.

\bibitem[Karlin, 1967]{karlin67central}
Karlin, S. (1967).
\newblock Central limit theorems for certain infinite urn schemes.
\newblock {\em J. Math. Mech.}, 17:373--401.

\bibitem[Kulik and Soulier, 2011]{kulik11tail}
Kulik, R. and Soulier, P. (2011).
\newblock The tail empirical process for long memory stochastic volatility
  sequences.
\newblock {\em Stochastic Process. Appl.}, 121(1):109--134.

\bibitem[Kulik and Soulier, 2012]{kulik12limit}
Kulik, R. and Soulier, P. (2012).
\newblock Limit theorems for long-memory stochastic volatility models with
  infinite variance: partial sums and sample covariances.
\newblock {\em Adv. in Appl. Probab.}, 44(4):1113--1141.

\bibitem[Kulik and Soulier, 2013]{kulik13estimation}
Kulik, R. and Soulier, P. (2013).
\newblock Estimation of limiting conditional distributions for the heavy tailed
  long memory stochastic volatility process.
\newblock {\em Extremes}, 16(2):203--239.

\bibitem[Kulik and Soulier, 2015]{kulik15heavy}
Kulik, R. and Soulier, P. (2015).
\newblock Heavy tailed time series with extremal independence.
\newblock {\em Extremes}, 18(2):273--299.

\bibitem[Lacaux and Samorodnitsky, 2016]{lacaux16time}
Lacaux, C. and Samorodnitsky, G. (2016).
\newblock Time-changed extremal process as a random sup measure.
\newblock {\em Bernoulli}, 22(4):1979--2000.

\bibitem[Leadbetter et~al., 1983]{leadbetter83extremes}
Leadbetter, M.~R., Lindgren, G., and Rootz{\'e}n, H. (1983).
\newblock {\em Extremes and related properties of random sequences and
  processes}.
\newblock Springer Series in Statistics. Springer-Verlag, New York.

\bibitem[Mikosch and Rezapour, 2013]{mikosch13stochastic}
Mikosch, T. and Rezapour, M. (2013).
\newblock Stochastic volatility models with possible extremal clustering.
\newblock {\em Bernoulli}, 19(5A):1688--1713.

\bibitem[Molchanov, 2017]{molchanov17theory}
Molchanov, I. (2017).
\newblock {\em Theory of random sets}, volume~87 of {\em Probability Theory and
  Stochastic Modelling}.
\newblock Springer-Verlag, London.
\newblock Second edition of [ MR2132405].

\bibitem[Molchanov and Strokorb, 2016]{molchanov16max}
Molchanov, I. and Strokorb, K. (2016).
\newblock Max-stable random sup-measures with comonotonic tail dependence.
\newblock {\em Stochastic Process. Appl.}, 126(9):2835--2859.

\bibitem[O'Brien et~al., 1990]{obrien90stationary}
O'Brien, G.~L., Torfs, P. J. J.~F., and Vervaat, W. (1990).
\newblock Stationary self-similar extremal processes.
\newblock {\em Probab. Theory Related Fields}, 87(1):97--119.

\bibitem[Pitman, 2006]{pitman06combinatorial}
Pitman, J. (2006).
\newblock {\em Combinatorial stochastic processes}, volume 1875 of {\em Lecture
  Notes in Mathematics}.
\newblock Springer-Verlag, Berlin.
\newblock Lectures from the 32nd Summer School on Probability Theory held in
  Saint-Flour, July 7--24, 2002, With a foreword by Jean Picard.

\bibitem[Resnick, 1987]{resnick87extreme}
Resnick, S.~I. (1987).
\newblock {\em Extreme values, regular variation, and point processes},
  volume~4 of {\em Applied Probability. A Series of the Applied Probability
  Trust}.
\newblock Springer-Verlag, New York.

\bibitem[Resnick, 2007]{resnick07heavy}
Resnick, S.~I. (2007).
\newblock {\em Heavy-tail phenomena}.
\newblock Springer Series in Operations Research and Financial Engineering.
  Springer, New York.
\newblock Probabilistic and statistical modeling.

\bibitem[Samorodnitsky, 2016]{samorodnitsky16stochastic}
Samorodnitsky, G. (2016).
\newblock {\em Stochastic processes and long range dependence}.
\newblock Springer, Cham, Switzerland.

\bibitem[Samorodnitsky and Taqqu, 1994]{samorodnitsky94stable}
Samorodnitsky, G. and Taqqu, M.~S. (1994).
\newblock {\em Stable non-{G}aussian random processes}.
\newblock Stochastic Modeling. Chapman \& Hall, New York.
\newblock Stochastic models with infinite variance.

\bibitem[Samorodnitsky and Wang, 2019]{samorodnitsky19extremal}
Samorodnitsky, G. and Wang, Y. (2019).
\newblock Extremal theory for long range dependent infinitely divisible
  processes.
\newblock {\em Ann. Probab.}, 47(4):2529--2562.

\bibitem[Sibuya, 1979]{sibuya79generalized}
Sibuya, M. (1979).
\newblock Generalized hypergeometric, digamma and trigamma distributions.
\newblock {\em Ann. Inst. Statist. Math.}, 31(3):373--390.

\bibitem[Stoev and Wang, 2019]{stoev19exchangeable}
Stoev, S. and Wang, Y. (2019).
\newblock Exchangeable random partitions from max-infinitely-divisible
  distributions.
\newblock {\em Statist. Probab. Lett.}, 146:50--56.

\bibitem[Vervaat, 1997]{vervaat97random}
Vervaat, W. (1997).
\newblock Random upper semicontinuous functions and extremal processes.
\newblock In {\em Probability and lattices}, volume 110 of {\em CWI Tract},
  pages 1--56. Math. Centrum, Centrum Wisk. Inform., Amsterdam.

\end{thebibliography}

\end{document}